\documentclass[11pt,leqno]{article}
\usepackage{amsmath}
\usepackage{amssymb}
\usepackage{theorem}
\usepackage{soul,color}
\definecolor{labelkey}{rgb}{0,0.08,0.45}
\definecolor{refkey}{rgb}{0,0.6,0.0}
\definecolor{Brown}{rgb}{0.45,0.0,0.05}
\definecolor{lime}{rgb}{0.00,0.8,0.0}
\definecolor{lblue}{rgb}{0.5,0.5,0.99}
\newcommand{\GX}{\ensuremath{\Gamma}}

\newcommand{\To}{\ensuremath{\rightrightarrows}}

\newcommand{\fenv}[1]%
{\ensuremath{\,\overrightarrow{\operatorname{env}}_{#1}}}
\newcommand{\benv}[1]%
{\ensuremath{\,\overleftarrow{\operatorname{env}}_{#1}}}
\newcommand{\emp}{\ensuremath{\varnothing}}

\newcommand{\scal}[2]{\left\langle{#1},{#2}  \right\rangle}

\newcommand{\HH}{\ensuremath{\mathcal H}}

\newcommand{\RR}{\ensuremath{\mathbb R}}

\newcommand{\ball}{\ensuremath{{\mathbb B}}}

\newcommand{\NN}{\ensuremath{\mathbb N}}

\newcommand{\dom}{\ensuremath{\operatorname{dom}}}
\newcommand{\argmin}{\ensuremath{\operatorname*{argmin}}}

\newcommand{\gr}{\ensuremath{\operatorname{gr}}}
\newcommand{\prox}{\ensuremath{\operatorname{Prox}}}

\newcommand{\inte}{\ensuremath{\operatorname{int}}}

\newcommand{\ran}{\ensuremath{\operatorname{ran}}}

\newcommand{\Fix}{\ensuremath{\operatorname{Fix}}}
\newcommand{\Id}{\ensuremath{\operatorname{Id}}}

\newcommand{\fettf}{\ensuremath{\boldsymbol{f}}}
\newcommand{\fettla}{\ensuremath{\boldsymbol{\lambda}}}
\newcommand{\pflm}{\ensuremath{p_\gamma({\fettf,\fettla)}}}
\newcommand{\pfl}{\ensuremath{p({\fettf,\fettla})}}

\newcommand{\fourIdx}[5]{%
   \setbox1=\hbox{\ensuremath{^{#1}}}%
   \setbox2=\hbox{\ensuremath{_{#2}}}%
   \setbox5=\hbox{\ensuremath{#5}}%
   \hspace{\ifnum\wd1>\wd2\wd1\else\wd2\fi}%
   \ensuremath{\copy5^{\hspace{-\wd1}\hspace{-\wd5}#1\hspace{\wd5}#3}%
                     _{\hspace{-\wd2}\hspace{-\wd5}#2\hspace{\wd5}#4}%
}}

\newcommand{\Yosida}[2]{\ensuremath{\fourIdx{\,#1\negthinspace}{}{}{}{\displaystyle #2}}}

\newcommand{\yosida}[1]{\ensuremath{\fourIdx{\,#1\negthinspace}{}{}{}{\displaystyle }}}

\newcommand{\menge}[2]{\big\{{#1} \mid {#2}\big\}}

{\begin{list}{}{%
\settowidth{\labelwidth}{\textrm{#1~}}%
\setlength{\leftmargin}{\labelwidth+\labelsep}}}
{\end{list}}
\newtheorem{theorem}{Theorem}[section]
\newtheorem{lemma}[theorem]{Lemma}
\newtheorem{corollary}[theorem]{Corollary}
\newtheorem{proposition}[theorem]{Proposition}
\newtheorem{definition}[theorem]{Definition}
\theoremstyle{plain}{\theorembodyfont{\rmfamily}
}
\theoremstyle{plain}{\theorembodyfont{\rmfamily}
}
\theoremstyle{plain}{\theorembodyfont{\rmfamily}
}
\theoremstyle{plain}{\theorembodyfont{\rmfamily}
\newtheorem{example}[theorem]{Example}}
\newtheorem{fact}[theorem]{Fact}
\theoremstyle{plain}{\theorembodyfont{\rmfamily}
\newtheorem{remark}[theorem]{Remark}}

\begin{document}

\title{\sffamily
Compositions and Averages of Two Resolvents:\\
Relative Geometry of Fixed Points Sets and \\
a Partial Answer to a Question by C.\ Byrne}

\author{
Xianfu Wang\thanks{Mathematics, Irving K.\ Barber School,
The University of British Columbia Okanagan, Kelowna,
B.C. V1V 1V7, Canada.
Email:
\texttt{shawn.wang@ubc.ca}.}
~~~and~
Heinz H.\ Bauschke\thanks{Mathematics, Irving K.\ Barber School,
The University of British Columbia Okanagan, Kelowna,
B.C. V1V 1V7, Canada. Email:
\texttt{heinz.bauschke@ubc.ca}.}
}

\date{March 24, 2010}

\maketitle

\vskip 8mm

\begin{abstract} \noindent
We show that the set of fixed points of the average
of two resolvents can be found from the
set of fixed points for compositions of two resolvents associated with scaled monotone operators.
Recently, the proximal average has attracted considerable
attention in convex analysis.
Our results imply that the minimizers of proximal-average
functions can be found from the set of fixed points for compositions of two
proximal mappings associated with scaled convex functions.
When both convex functions in the proximal average
are indicator functions of convex sets,
least squares solutions can be completely
recovered from the limiting cycles given by compositions of two projection
mappings. This provides a partial answer to a question posed by C.~Byrne.
A novelty of our approach is to use the notion of resolvent average and proximal average.
\end{abstract}

{\small
\sethlcolor{lblue}

}

\section{Introduction}
Throughout, $\HH$ is a Hilbert space with inner
product $\scal{\cdot}{\cdot}$ and induced norm $\|\cdot\|$,
and $\Gamma(\HH)$ is the set of proper lower semicontinuous convex functions on $\HH$.
 Let
$A:\HH\To 2^{\HH}$ be a set-valued operator with graph
 $\gr A :=\menge{(x,u)\in\HH\times\HH}{u\in Ax}$.
The set-valued inverse $A^{-1}$ of $A$ has graph
 $\menge{(u,x)\in\HH}{u\in Ax}$, and the resolvent of $A$
is $J_{A}:=(A+\Id)^{-1}$ where $\Id:\HH\to\HH$
denotes the identity mapping.
The operator $A$ is monotone if $\scal{x-y}{u-v}\geq 0$ for all $(x,u), (y,v)\in\gr A$; $A$ is maximal monotone if $A$ is monotone and no proper enlargement of $\gr A$ is monotone.

Let $A_{1},A_{2}$ be two maximal monotone operators, and $\lambda_{1}+\lambda_{2}=1$ with $\lambda_{i}>0$. The \emph{resolvent average} of
$A_{1},A_{2}$ with weights $\lambda_{1},\lambda_{2}$ is defined by
$$A:=[\lambda_{1}J_{A_{1}}+\lambda_{2}J_{A_{2}}]^{-1}-\Id,$$
and it owes its name to the identity
$$J_{A}=\lambda_{1}J_{A_{1}}+\lambda_{2}J_{A_{2}}.$$

\emph{This paper is concerned with the relationships among the
fixed point sets of the resolvent average $J_{A}$,  the
resolvent compositions $J_{A_{1}/\lambda_{2}}J_{A_{2}/\lambda_{1}}$
and $J_{A_{2}/\lambda_{1}}J_{A_{1}/\lambda_{2}}$.
Although there appears to be no clear relationships between
the fixed point sets
of $\Fix (\lambda_{1}J_{A_{1}}+\lambda_{2}J_{A_{2}})$, and of
$\Fix J_{A_{1}}J_{A_{2}}$ and $\Fix J_{A_{2}}J_{A_{1}}$,
we will observe that $\Fix (\lambda_{1}J_{A_{1}}+\lambda_{2}J_{A_{2}})$
can be completely recovered from $\Fix (J_{A_{1}/\lambda_{2}}J_{A_{2}/\lambda_{1}})$ or $\Fix(J_{A_{2}/\lambda_{1}}J_{A_{1}/\lambda_{2}})$.}

Our investigation relies on the resolvent average and proximal average, \cite{Baus08,lucet,Baus09,moffat}.
Although compositions of resolvents (even more generally 
strongly nonexpansive mappings) have been studied
\cite{Baus05,comb,patrick2,cheney,bruck,Baus96,Baus04,luke},
the connections between the fixed point set of
compositions and the fixed point set of the average of two
resolvents appear to be new,
\emph{even when specialized to projection operators}.

The paper is organized as follows. Section~\ref{resultsknown} gathers some
known facts used in later sections. In Section~\ref{resolvcomp} we
concentrate on resolvents. In order to find zeros of the resolvent
average, we consider several inclusion problems. It turns out that their solution sets can be characterized in terms of fixed point sets
associated with the resolvent average and with resolvent compositions.
We provide homeomorphisms
among these fixed point sets.
In Section~\ref{proxcomp} we apply --- and refine ---
the results of Section~\ref{resolvcomp} to proximal mappings.
The inclusion problems now translate into finding minimizers of
proximal averages of convex functions.
The Yosida regularization is the key tool for monotone operators,
and its role is played by the Moreau envelope for convex functions.
When specialized to projections, our results say that the least square solutions can be completed recovered from the solutions of
alternating projections. This answers one of the
question posed in \cite[page~305]{charles} by Byrne for two sets, while the question for more than two sets is still open.
In  Section~\ref{sanitycheck} we give three
examples to illustrate our results. They illustrate that
a recovery of $\Fix (\lambda_{1}J_{A_{1}}+\lambda_{2}J_{A_{2}})$
from $\Fix (J_{A_{1}}J_{A_{2}})$ and $\Fix (J_{A_{2}}J_{A_{1}})$
seems impossible.

Our notation is standard and follows, e.g., \cite{Rock70,Rock98,Simons}.
For a monotone operator $A:\HH\To 2^{\HH}$, the sets
 $\dom A :=\menge{x\in\HH}{Ax\neq\emp},
 \ran A :=\menge{u\in \HH}{(\exists x \in \HH) \ u\in Ax}$
 are the domain, range of $A$ respectively.
It will be convenient to write
$\widetilde{A}:=(-\Id)\circ A^{-1}\circ (-\Id).$
The Yosida approximation of
$A$ of index $\gamma\in (0,+\infty)$ is given by
\begin{equation}\label{regular}
\Yosida{\gamma}{A} :=(\Id-J_{\gamma A})/\gamma=(\gamma\Id+A^{-1})^{-1}.
\end{equation}
For a mapping $T: D\to \HH$, where $D\subseteq \HH$,
the fixed point set of
$T$ will be denoted by $\Fix T:=\menge{x\in\HH}{Tx=x}$. A mapping $T$ between metric spaces $X$ and $Y$ is called a \emph{homeomorphism} if $T$ is a bijection (i.e., one-to-one and onto), $T$ is continuous and its inverse $T^{-1}$ is also continuous.
For a sequence $(x_{n})_{n\in\NN}$ of $\HH$, $x_{n}\rightharpoonup x\in\HH$ means that $(x_{n})_{n\in\NN}$
converges weakly to $x$.

For a proper lower semicontinuous function $f\in\GX(\HH)$,
the subdifferential operator  $\partial f: \HH\To\HH$ of $f$ which is
given by
$x\mapsto\partial f(x) :=\menge{x^*\in\HH}{f(y)\geq f(x)+\langle x^*,y-x\rangle \ \forall y\in \HH}$ is maximal monotone.
The resolvent of $\partial f$ is called the \emph{proximal mapping} of $f$, i.e., $\prox_{f}:=J_{\partial f}$.
Note that $\prox_{f}$ has a full domain.
Also, $f^*$ denotes the Fenchel conjugate of $f$,
i.e., $ (\forall x^*\in \HH) \ f^*(x^*):=\sup_{x}\big(\langle x^*,x\rangle-f(x)\big).$
The \emph{Moreau envelope} of $f$ with parameter $\gamma$ is given by
$$e_\gamma f(x) :=\inf_{y}\bigg(f(y)+\frac{1}{2\gamma}\|x-y\|^2\bigg)\quad \text{ for every $x\in\HH$}.$$
The domain of $f$ will be denoted by $\dom f$. For $f_{1},f_{2}\in \GX(\HH)$, $f_1\oplus f_{2}$ means
$(f_1\oplus f_{2})(x,y):=f_{1}(x)+f_{2}(y)$ for all $x,y\in\HH$. We let $j(x):=\|x\|^2/2$ for every $x\in\HH$
and we will use $j$ and $\|\cdot\|^2/2$ interchangeably.
For a subset $C\subseteq \HH$, the indicator function is defined by $\iota_{C}(x)=0$ if $x\in C$ and
$+\infty$ otherwise. We use $d_{C}(x):=\inf\{\|x-y\||\ y\in C\}$ for every
$x\in\HH$ for the distance function, $P_{C}:=\prox_{\iota_{C}}$ for the
projection on set $C$, $N_{C}:=\partial \iota_{C}$ for the normal cone
operator, and $\inte C$ for the interior of the set $C$.

\section{Auxiliary results and facts}\label{resultsknown}
We gather some facts on strongly nonexpansive mapping, on the
proximal point algorithm, and on
fixed point sets of compositions of two resolvents.

\begin{definition} Let $T:D\to\HH$, where $D\subseteq\HH$.
We say that
\begin{enumerate}
\item $T$ is \emph{nonexpansive} if
$$\|Tx-Ty\|\leq \|x-y\|\quad \forall x,y\in D;$$
\item $T$ is \emph{strongly nonexpansive} if $T$ is nonexpansive and
$(x_{n}-y_{n})-(Tx_{n}-Ty_{n})\rightarrow 0$ whenever $(x_{n})_{n\in \NN}, (y_{n})_{n\in \NN}$ are
sequences in $D$ such that $(x_{n}-y_{n})_{n\in \NN}$ is bounded and $\|x_{n}-y_{n}\|-\|Tx_{n}-Ty_{n}\|
\rightarrow 0$;
\item $T$ is \emph{firmly nonexpansive} if
$$\|Tx-Ty\|^2\leq \scal{Tx-Ty}{x-y}\quad \forall x, y\in D;$$
\item $T$ is \emph{attracting} if $T$ is nonexpansive and for every $x\not\in\Fix T, y\in\Fix T$ one has
$$\|Tx-Ty\|<\|x-y\|.$$
\end{enumerate}
\end{definition}
The following fact is well-known.
\begin{fact}\label{elementarym}
Let $B:\HH\To 2^{\HH}$ be monotone operator and $\gamma>0$. Then
\begin{enumerate}
\item $B^{-1}(0)=\Fix (J_{\gamma B})=(\Yosida{\gamma}{B})^{-1}(0)$.
\item $J_{B}$ is firmly nonexpansive.
\item $J_{B}$ has a full domain if and only if $B$ is maximal monotone.
\end{enumerate}
\end{fact}
\begin{proof}
(i). This may be readily verified using definitions involved.
(ii) and (iii): See \cite{minty} or \cite[Fact 6.2, Corollary 6.3]{bartz}.
\end{proof}

\begin{fact}[Bruck \& Reich \cite{bruck}]\label{bruckreich}
Let $T$, and $(T_{i})_{1\leq i\leq m}$ be operators from $\HH$ to $\HH$. Then the following properties hold:
\begin{enumerate}
\item If $T$ is firmly nonexpansive, then it is strongly nonexpansive.
\item If the operators $(T_{i})_{1\leq i\leq m}$ are strongly nonexpansive, then the composition
$T_{1}\cdots T_{m}$ is also strongly nonexpansive.
\item If $T_{1}$ is strongly nonexpansive and $T_{2}$ is nonexpansive and $0<c<1$, then
$S=(1-c)T_{1}+cT_{2}$ is strongly nonexpansive.
\item Suppose that $T$ is strongly nonexpansive and let $x_{0}\in\HH$. If $\Fix T\neq \emp$, then
the sequence $(T^{n}x_{0})_{n}$ converges weakly to some point in $\Fix T$; otherwise, $\|T^nx_{0}\|\rightarrow\infty$.
\end{enumerate}
\end{fact}
Fact~\ref{elementarym}(ii) and Fact~\ref{bruckreich} immediately 
give the following result.
\begin{corollary} Let $A_{1}, A_{2}:\HH\To\HH$ be maximal monotone operators. For $x_{0}\in\HH$
let $(x_{n})_{n\in \NN}$ be generated by
 $$(\forall n \in \NN)\quad  x_{n+1}=J_{A_{1}}J_{A_{2}}x_{n};$$
 For $y_{0}\in \HH$ let $(y_{n})_{n\in\NN}$ be generated by
 $$(\forall n\in \NN) \quad y_{n+1}=J_{A_{2}}J_{A_{1}}y_{n}.$$
If $\Fix J_{A_{1}}J_{A_{2}}\neq\emp$, then $(x_{n})$ converges weakly to some point of $\Fix J_{A_{1}}J_{A_{2}}$, and $(y_{n})$ converges weakly to some point of $\Fix J_{A_{2}}J_{A_{1}}$.
\end{corollary}

\begin{fact}[Rockafellar \cite{rockprox}]\label{terry}
Let $A:\HH\To \HH$ be maximal monotone.
Assume that $A^{-1}(0)\neq\emp$. For any starting point $x_{0}$, the sequence $(x_{n})$ generated by
the proximal point algorithm
$$x_{n+1}=J_{A}(x_{n})=(\Id+A)^{-1}(x_{n})$$
converges weakly to a point in $A^{-1}(0)$ and
$x_{n+1}-x_{n}\to 0$.
\end{fact}

Let $R$ denote the ``transpose'' mapping on $\HH\times\HH$,
namely $R:\HH\times \HH\to \HH\times \HH: (x,y)\mapsto (y,x)$.

\begin{fact} \emph{(See \cite{Baus05}.)}\label{key1}
 Let $A, B:\HH\To \HH$ be maximal monotone operators and
 $\gamma\in (0, +\infty)$.
 Set
\begin{align*}
S & :=(\Id-R+\gamma (A\times B))^{-1}(0,0).\\
S^* & :=[(\Id-R)^{-1}+(A^{-1}\times B^{-1})\circ (\Id/\gamma)]^{-1}(0,0).\\
E& :=(A+\yosida{\gamma}B)^{-1}(0),\quad F:=(B+\yosida{\gamma}A)^{-1}(0).\\
u^*& : =J_{(A^{-1}+\widetilde{B})/\gamma}(0) \text{ and } v^*:=J_{(\widetilde{A}+B^{-1})/\gamma}(0).
\end{align*}
Then $S\neq\varnothing$ 
$\Leftrightarrow$ $S^*\neq\varnothing$
$\Leftrightarrow$ $E\neq\varnothing$
$\Leftrightarrow$ $F\neq \varnothing$
$\Leftrightarrow$ $u^*$ is well defined
$\Leftrightarrow$ $v^*$ is well defined,
in which case the following hold.
\begin{enumerate}
\item $E=\Fix J_{\gamma A}J_{\gamma B}=J_{\gamma A}(F)$ and $F=\Fix J_{\gamma B}J_{\gamma A}
=J_{\gamma B}(E)$.
\item $S=\Fix J_{\gamma(A\times B)}R=(E\times F)\cap\gr J_{\gamma B}$.
\item $S^*=\{(\gamma u^*,\gamma v^*)\}$ and $u^*=-v^*$.
\item $S^*=(R-\Id)(S)$.
\item $J_{\gamma B}|_{E}:E\to F: x\mapsto  x+\gamma u^*$ is a bijection with inverse
$J_{\gamma A}|_{F}: F\to E: y\mapsto y+\gamma v^*$.
\item $E=A^{-1}(u^*)\cap (\yosida{\gamma}B)^{-1}(v^*)$
and $F=(\yosida{\gamma}A)^{-1}(u^*)\cap B^{-1}(v^*).$
\item $S=(E\times F)\cap (R-\Id)^{-1}(S^*)$.
\end{enumerate}
\end{fact}

\begin{fact} \emph{(See, e.g., \cite[Propositions 2.10, 2.12]{Baus96}.)} \label{onesetonly}
Assume that $T_{1}, T_{2}$ are attracting and
$\Fix T_{1}\cap \Fix T_2\neq\varnothing$.
Let $\lambda_{1}+\lambda_{2}=1$, with each $\lambda_{i}>0$. Then
$$\Fix(\lambda_1 T_{1}+\lambda_{2}T_{2})=\Fix T_{1}\cap \Fix T_{2}=\Fix(T_{1}\circ T_{2})
=\Fix(T_{2}\circ T_{1}).$$
\end{fact}
The class of attractive mappings properly contains the class of strongly nonexpansive mappings. See also
\cite[Lemma 2.1]{bruck} for results related to Fact~\ref{onesetonly}.

The following two facts relate the solutions of primal problems
to the solutions of certain dual problems.
For functions, a constraint qualification is needed; however,
for monotone operators, the ensuing duality requires no
constraint qualification.

\begin{fact}[Fenchel-Rockafellar duality \cite{Rock98,Zalinescu}]\label{duality}
Assume that $f,g\in \GX(\HH)$ and $L:\HH\rightarrow\HH$ is a continuous linear operator. Suppose
there exists $x_{0}\in\dom f\cap L^{-1}(\dom g)$ such that $g$ is continuous at $Lx_{0}$. Then
$$\inf_{x\in\HH}\big(f(x)+g(Lx)\big)=-\min_{y^*\in\HH}
\big(f^*(-L^*y^*)+g^*(y^*)\big).$$
Furthermore, $\bar{x}$ is a minimizer  for $f+g\circ L$ if and only if there exists $\bar{y}^*\in\HH$ such that
$$-L^*\bar{y}^*\in\partial f(\bar{x}),\quad \bar{y}^*\in \partial g(L\bar{x}).$$
\end{fact}
\begin{fact}[Attouch-Th\'era duality \cite{attouch}]\label{dualityat}
Let $A,B:\HH\To 2^{\HH}$ be maximal monotone operators. Let
$S$ be the solution set of the primal problem
$$\text{ find $x\in\HH$ such that } 0\in Ax+Bx.$$
Let $S^*$ be the solution set of the dual problem
\begin{equation}
\label{theradual1}
\text{ find $x^*\in\HH$ such that } 0\in A^{-1}x^*+\widetilde{B}(x^*).
\end{equation}
Then
\begin{enumerate}
\item $S=\menge{x\in \HH}{(\exists\ x^*\in S^*)\ x^*\in Ax \text{
and } -x^*\in Bx}$.
\item $S^*=\menge{x^*\in\HH}{(\exists\ x\in S)\ x\in A^{-1}x^*
\text{ and } -x\in \widetilde{B}(x^*)}.$
\end{enumerate}
Moreover, let $S^*_{1}$ be the solution to the dual problem given by
\begin{equation}\label{theradual2}
\text{ find $y^*\in\HH$ such that } 0\in \widetilde{A}(y^*)+B^{-1}(y^*).
\end{equation}
Then $S^*_{1}=-S^*$. Consequently, up to a `$-$' sign change in the dual
variable, the Attouch-Th\'era duals
\eqref{theradual1} and \eqref{theradual2}
have the same solutions.
\end{fact}

The last result recorded in this section concerns 
basic properties of the resolvent average.

\begin{fact}[resolvent average]\label{knownlong}
Let $A_{1},A_{2}: \HH\To 2^{\HH}$ 
be maximal monotone operators, let $\lambda_1,\lambda_2,\gamma>0$ with
$\lambda_1+\lambda_2=1$, 
and set 
$$A:=\frac{(\lambda_{1}J_{\gamma A_{1}}+\lambda_{2}J_{\gamma A_{2}})^{-1}-\Id}{\gamma}.$$
Then
\begin{enumerate}
\item $J_{\gamma A}=\lambda_{1}J_{\gamma A_{1}}+\lambda_{2}J_{\gamma A_{2}}$ and
$\Yosida{\gamma}{A}=\lambda_{1}\Yosida{\gamma}{A_{1}}+\lambda_{2}
\Yosida{\gamma}{A_{2}}.$
\item $A$ is maximal monotone.
\end{enumerate}
\end{fact}
\begin{proof} (i) follows from the definitions involved. (ii): By Fact~\ref{elementarym}(iii)
and maximal monotonicity
of $A_{i}$, $J_{\gamma A_{i}}$ is firmly nonexpansive and has a full domain so that $J_{\gamma A}$ is firmly nonexpansive and has a full domain.
Then by Fact~\ref{elementarym}(iii)
again $\gamma A$ is maximal monotone, so is $A$.
\end{proof}

\section{Fixed points of resolvent average and compositions}\label{resolvcomp}

In this Section, we assume that
$A_{1},A_{2}:\HH\To 2^{\HH}$ are maximal monotone operators, and that
$\lambda_{1}+\lambda_{2}=1$ where each $\lambda_{i}>0$.

\subsection{Inclusion problem formulations and their 
common Attouch-Th\'era dual}

Consider the inclusion problems
\begin{align}
&\text{($R_{\gamma}$)\quad find $z$ such that}\quad  0\in
\Big(\big(\lambda_{1}J_{\gamma A_{1}}+\lambda_{2}J_{\gamma
A_{2}}\big)^{-1}-\Id \Big)(z);\label{sunday-1}\\[+ 2mm]
& \text{($P_{\gamma}$)\quad find $z$ such that}\quad  0= \big(\lambda_{1}
\Yosida{\gamma}{A_{1}}
+\lambda_{2}\Yosida{\gamma}{A_{2}}\big)(z);\label{sunday0}\\[+ 2mm]
& \text{($P$)\quad find $(x,y)$ such that}\quad  (0,0)\in \bigg(\frac{(\Id
-R)}{\gamma}+\Big(\frac{A_{1}}{\lambda_{2}}\times
\frac{A_{2}}{\lambda_{1}}\Big)\bigg)(x,y);\label{sunday1}\\[+2 mm]
&\text{find $x$ such that}\quad 0\in \bigg(\frac{A_{1}}{\lambda_{2}}+
\Yosida{\gamma}{\Big(\frac{A_{2}}{\lambda_{1}}\Big)}\bigg)(x);\label{sunday2}
\\[+2 mm]
&\text{find $y$ such that} \quad
0\in\bigg(\Yosida{\gamma}{\Big(\frac{A_{1}}{\lambda_{2}}\Big)}+\frac{A_{2}}{\lambda_{1}}\bigg)(y).\label{sunday3}
\end{align}

\begin{theorem}
Problems~\eqref{sunday-1}--\eqref{sunday3} are equivalent
in the sense that
if one of the problems possesses a solution, then so do all the others.
\end{theorem}
\begin{proof}
\eqref{sunday-1}$\Leftrightarrow$\eqref{sunday0}:
$z$ solves \eqref{sunday-1} if and only if $z$ solves
$0\in A(z)$ where
$$A=\frac{\big(\lambda_{1}J_{\gamma A_{1}}+\lambda_{2}J_{\gamma A_{2}}\big)^{-1}-\Id}{\gamma}.$$
It suffices to apply Fact~\ref{elementarym}(i) and Fact~\ref{knownlong}(i) to $A$.

\eqref{sunday-1}$\Leftrightarrow$\eqref{sunday1}: Note that
$z$ solves \eqref{sunday-1} if and only if
$z=\lambda_{1}J_{\gamma A_{1}}(z)+\lambda_{2}J_{\gamma A_{2}}(z)$.
Let $J_{\gamma A_{1}}(z)=x$, $J_{\gamma A_{2}}(z)=y$. We have
$z$ solves \eqref{sunday-1} if and only if
\begin{equation}\label{clearform}
\begin{cases}
& z=\lambda_{1}x+\lambda_{2}y\\
& x=J_{\gamma A_{1}}(z)\\
& y=J_{\gamma A_{2}}(z).
\end{cases}
\end{equation}
We claim that $(x,y)$ solves \eqref{sunday1}. Indeed,
\eqref{clearform} gives $z\in \gamma A_{1}(x)+x$, $z\in \gamma A_{2}(y)+y$, i.e.,
\begin{align}\label{splittedform}
0 &\in \gamma A_{1}(x)+(x-z)\\
0 & \in \gamma A_{2}(y)+(y-z).
\end{align}
As $z=\lambda_{1}x+\lambda_2 y$, we have $x-z=\lambda_{2}(x-y)$, $y-z=\lambda_{1}(y-x)$. Therefore,
\eqref{splittedform} gives
\begin{align}
0 &\in \gamma A_{1}(x)+\lambda_{2}(x-y)\\
0 & \in \gamma A_{2}(y)+\lambda_{1}(y-x),
\end{align}
equivalently,
\begin{align}
0 &\in \frac{\gamma A_{1}(x)}{\lambda_{2}}+(x-y)\\
0 & \in \frac{\gamma A_{2}(y)}{\lambda_{1}}+(y-x).
\end{align}
In the product space setting,
\begin{equation}
(0,0)\in \bigg((\Id-R)(x,y)+\frac{\gamma A_{1}(x)}{\lambda_{2}}\times \frac{\gamma A_{2}(y)}{\lambda_{1}}\bigg).
\end{equation}
 Dividing  both sides by $\gamma$ gives
\begin{equation}\label{productform}
(0,0)\in \bigg(\frac{(\Id-R)}{\gamma}(x,y)+\frac{ A_{1}(x)}{\lambda_{2}}\times \frac{ A_{2}(y)}{\lambda_{1}}\bigg),
\end{equation}
as required.
Conversely, let $(x,y)$ solves \eqref{sunday1}. Put $z=\lambda_{1}x+\lambda_{2}y$. Exactly reverse
the arguments from \eqref{productform} to \eqref{splittedform} to get \eqref{clearform}. Hence $z$ solves \eqref{sunday-1}.

\eqref{sunday1}$\Leftrightarrow$\eqref{sunday2}: $(x,y)$ solves \eqref{sunday1} if and only if
\begin{align}
0 &\in \frac{\gamma A_{1}(x)}{\lambda_{2}}+(x-y)\label{thisone}\\
0 & \in \frac{\gamma A_{2}(y)}{\lambda_{1}}+(y-x).\label{thistwo}
\end{align}
From \eqref{thistwo},
$y=J_{\gamma A_{2}/\lambda_{1}}(x)$. Put this in \eqref{thisone} to get
\begin{equation}\label{thisthree}
0 \in \frac{\gamma A_{1}(x)}{\lambda_{2}}+x-J_{\gamma A_{2}/\lambda_{1}}(x).
\end{equation}
Dividing both sides by $\gamma$ gives
$$0\in \frac{A_{1}(x)}{\lambda_{2}}+\frac{x-J_{\gamma A_{2}/\lambda_{1}}(x)}{\gamma}
=\frac{A_{1}(x)}{\lambda_{2}}+\Yosida{\gamma}{\Big(\frac{A_{2}}{\lambda_{1}}\Big)}(x),$$
which says that $x$ solves \eqref{sunday2}.
Conversely, $x$ solves  \eqref{sunday2} if and only if \eqref{thisthree} holds.
Put $y=J_{\gamma A_{2}/\lambda_{1}}(x)$. Then $x\in \gamma A_{2}(y)/\lambda_{1}+y$, and \eqref{thisthree}
gives
$0 \in \frac{\gamma A_{1}(x)}{\lambda_{2}}+x-y.$ Hence $(x,y)$ satisfies \eqref{thisone} and \eqref{thistwo}.

As in \eqref{sunday1}$\Leftrightarrow$\eqref{sunday2},  similarly one can show \eqref{sunday1}$\Leftrightarrow$\eqref{sunday3}.
\end{proof}

We proceed to show that all of them share one common Attouch-Th\'era dual problem.

\begin{theorem}\label{operator}
 Up to a `$-$' sign change of variable, the following inclusion problems
 have the same Attouch-Th\'era dual solution.
\begin{enumerate}
\item $\text{($P_{\gamma}$)\quad find $z$ such that}\quad
0\in \left(\frac{\Yosida{\gamma}{A_{1}}}{\displaystyle\lambda_{2}}
+\frac{\Yosida{\gamma}{A_{2}}}{\displaystyle \lambda_{1}}\right)(z);$
\item $\text{(P)\quad find $(x,y)$ such that}\quad
(0,0)\in \left(\frac{(\Id -R)}{\gamma}+\big(\frac{A_{1}}{\lambda_{2}}\times
\frac{A_{2}}{\lambda_{1}}\big)\right)(x,y);$
\item $\text{find $x$ such that}\quad
0\in \left(\frac{A_{1}}{\lambda_{2}}+
\Yosida{\gamma}{\big(\tfrac{A_{2}}{\lambda_{1}}\big)}\right)(x);$
\item $\text{find $y$ such that} \quad
0\in\left(\Yosida{\gamma}{\big(\tfrac{A_{1}}{\lambda_{2}}\big)}+
\frac{A_{2}}{\lambda_{1}}\right)(y).$
\end{enumerate}
Namely, up to a `$-$' sign change in the dual variable, their
Attouch-Th\'era dual has the form
\begin{equation}
\text{ find $z^*$ such that } \quad  0\in\gamma z^*+(A_{1}/\lambda_{2})^{-1}(z^*)+\widetilde{A_{2}/\lambda_{1}}(z^*).
\end{equation}
Moreover, the set of solutions is either empty or a singleton.
\end{theorem}
\begin{proof}
By \eqref{regular}, we have
 $$\widetilde{\Yosida{\gamma}{A}}=
-(\Yosida{\gamma}{A})^{-1}(-\Id)=\gamma\Id +\widetilde{A},$$
$$\widetilde{A/\lambda}=\widetilde{A}(\lambda\Id).$$
(i). The Attouch-Th\'era dual is:
$$0 \in \left[(\Yosida{\gamma}{A_{1}}/\lambda_{2})^{-1} +
\widetilde{\Yosida{\gamma}{A_{2}}/\lambda_{1}}\right](z^*).$$
We have \begin{align*}
\big(\Yosida{\gamma}{A_{1}}/\lambda_{2}\big)^{-1} +
\widetilde{\Yosida{\gamma}{A_{2}}/\lambda_{1}}
&= (\Yosida{\gamma}{A_{1}})^{-1}(\lambda_{2}\Id)
+\widetilde{\Yosida{\gamma}{A_{2}}}(\lambda_{1}\Id)\\
&=(\gamma\Id+A_{1}^{-1})(\lambda_{2}\Id)+(\gamma\Id+\widetilde{A_{2}})(\lambda_{1}\Id)\\
&=\gamma\Id+A_{1}^{-1}(\lambda_{2}\Id)+\widetilde{A_{2}}(\lambda_{1}\Id)\\
&=\gamma\Id +(A_{1}/\lambda_{2})^{-1}+\widetilde{A_{2}/\lambda_{1}}.
\end{align*}
Hence the dual is
\begin{equation}\label{lastday}
0\in [\gamma\Id +(A_{1}/\lambda_{2})^{-1}+\widetilde{A_{2}/\lambda_{1}}](z^*).
\end{equation}

(ii). The Attouch-Th\'era dual is
\begin{equation}\label{fun}
(0,0)\in \bigg[\gamma(\Id-R)^{-1}+\bigg(\frac{A_{1}}{\lambda_{2}}
\times\frac{A_{2}}{\lambda_{1}}\bigg)^{-1}\bigg](x^*,y^*).\end{equation}
Since $\ran (\Id-R)=\menge{(d,-d)}{d\in \HH},$
we have $y^*=-x^*$. \eqref{fun} reduces to find $x^*$ such that
$$(0,0)\in \bigg[\gamma(\Id-R)^{-1}+\bigg(\frac{A_{1}}{\lambda_{2}}
\times\frac{A_{2}}{\lambda_{1}}\bigg)^{-1}\bigg](x^*,-x^*).$$
Then $x^*$ solves the dual if and only if
there exists $y\in \HH$ such that
$$(0,0)\in \gamma (y+x^*,y)+\bigg(\frac{A_{1}}{\lambda_{2}}\bigg)^{-1}(x^*)
\times\bigg(\frac{A_{2}}{\lambda_{1}}\bigg)^{-1}(-x^*),
$$
which transpires to
$$0\in \gamma (y+x^*)+\bigg(\frac{A_{1}}{\lambda_{2}}\bigg)^{-1}(x^*),
\quad 0\in \gamma y+\bigg(\frac{A_{2}}{\lambda_{1}}\bigg)^{-1}(-x^*).$$
This is equivalent to find $x^*$ such that
$$0\in \gamma x^*+\bigg(\frac{A_{1}}{\lambda_{2}}\bigg)^{-1}(x^*)-\bigg(\frac{A_{2}}{\lambda_{1}}\bigg)^{-1}(-x^*)
=\left[\gamma \Id+
(A_{1}/\lambda_{2})^{-1}+\widetilde{A_{2}/\lambda_{1}}\right](x^*).$$

(iii). The Attouch-Th\'era dual is
$$0\in \bigg(\frac{A_{1}}{\lambda_{2}}\bigg)^{-1}(x^*)+
\widetilde{\yosida{\gamma}\bigg(\frac{A_{2}}{\lambda_{1}}\bigg)}(x^*).
$$
The right-hand side becomes
\begin{align*}
&\bigg(\frac{A_{1}}{\lambda_{2}}\bigg)^{-1}+
\gamma\Id + \widetilde{\bigg(\frac{A_{2}}{\lambda_{1}}\bigg)}=\gamma\Id +(A_{1}/\lambda_{2})^{-1}+\widetilde{A_{2}/\lambda_{1}}
\end{align*}
Hence the dual is
$$0\in \left[\gamma\Id
+(A_{1}/\lambda_{2})^{-1}+\widetilde{A_{2}/\lambda_{1}}\right](x^*).
$$
(iv). The Attouch-Th\'era dual is
$$0\in \widetilde{\yosida{\gamma}\big(A_{1}/\lambda_{2}\big)}(y^*)
+(A_{2}/\lambda_{1})^{-1}(y^*).$$
We have
\begin{align*}
\widetilde{\yosida{\gamma}\big(A_{1}/\lambda_{2}\big)}
+(A_{2}/\lambda_{1})^{-1}
=\gamma\Id +\widetilde{A_{1}/\lambda_{2}}+(A_{2}/\lambda_{1})^{-1}.
\end{align*}
Then the dual becomes
$$0\in \gamma y^*+\widetilde{A_{1}/\lambda_{2}}(y^*)+(A_{2}/\lambda_{1})^{-1}(y^*),$$
that is,
$$0\in \gamma y^*-(A_{1}/\lambda_{2})^{-1}(-y^*)+(A_{2}/\lambda_{1})^{-1}(y^*).$$
Multiplying both sides by $-1$, followed by making the substitution
$z^*=-y^*$, we obtain
$$0\in\gamma z^*+(A_{1}/\lambda_{2})^{-1}(z^*)+
\widetilde{A_{2}/\lambda_{1}}(z^*).$$
The proof is complete.
\end{proof}

\subsection{Characterization of solution sets}

Problem~\eqref{sunday1} has its Attouch-Th\'era dual given by
\begin{equation}
\text{(D)\quad find $(x^*,y^*)$ such that} \quad
(0,0)\in \bigg(\left(\frac{\Id-R}{\gamma}\right)^{-1}+
\left(\frac{A_{1}}{\lambda_{2}}
\times\frac{A_{2}}{\lambda_{1}}\right)^{-1}\bigg)(x^*,y^*).
\end{equation}
The following result gives a fixed point characterization to the solution sets of \eqref{sunday-1}--\eqref{sunday3} when $\gamma =1$.
\begin{theorem}\label{evening}
The following assertions hold.
\begin{enumerate}
\item \label{set1} \emph{({\bf Fixed points of resolvent average})} Let $A=(\lambda_{1}J_{A_{1}}+\lambda_{2}J_{A_{2}})^{-1}-\Id$. Then
$$\Fix J_{A}=\big(\lambda_{1}\Yosida{1}{A_{1}} +
\lambda_{2}\Yosida{1}{A_{2}}\big)^{-1}(0)
=\menge{z\in\HH}{z=J_{A}(z)=\lambda_{1}J_{A_{1}}(z)+\lambda_{2}J_{A_{2}}(z)}.$$
\item \label{set2}\emph{({\bf Fixed points of compositions})}
Set
$E:=\Big(\frac{A_{1}}{\lambda_{2}}+
\Yosida{1}{\big(\tfrac{A_{2}}{\lambda_{1}}\big)}\Big)^{-1}(0)$.
Then $E =\Fix J_{A_{1}/\lambda_{2}}J_{A_{2}/\lambda_{1}}
=J_{A_{1}/\lambda_{2}}(F).$
\item \label{set3}\emph{({\bf Fixed points of compositions})} Set
$F:=\Big(\Yosida{1}{\big(\tfrac{A_{1}}{\lambda_{2}}\big)}+\frac{A_{2}}{\lambda_{1}}\Big)^{-1}(0)$. Then $F
    =\Fix J_{A_2/\lambda_{1}}J_{A_{1}/\lambda_{2}}=J_{A_{2}/\lambda_{1}}(E)$.
\item \label{set4}\emph{({\bf Fixed points of alternating resolvents})}
Set
$S:=\left((\Id -R)+\big(\frac{A_{1}}{\lambda_{2}}\times
\frac{A_{2}}{\lambda_{1}}\big)\right)^{-1}(0,0)$. Then
$$
S=\{(x,y)| x=J_{A_{1}/\lambda_{2}}y, y=J_{A_{2}/\lambda_{1}}x\}=\Fix \big(J_{A_{1}/\lambda_{2}\times A_{2}/\lambda_{1}}\circ R\big)
 =(E\times F)\cap \gr J_{A_{2}/\lambda_{1}}.
$$
\item \label{evening4}
Set $S^*=\left((\Id-R)^{-1}+\big(\frac{A_{1}}{\lambda_{2}}
\times\frac{A_{2}}{\lambda_{1}}\big)^{-1}\right)^{-1}(0,0).$ Then $S^*$ is at most a singleton with
$$S^*=\menge{(u^*,v^*)}{u^*=J_{(A_{1}/\lambda_{2})^{-1}+\widetilde{A_{2}/\lambda_{1}}}(0),
v^*=J_{\widetilde{A_{1}/\lambda_{2}}+(A_{2}/\lambda_{1})^{-1}}(0)}.$$
Moreover, $u^*=-v^*$. (Note that $S^*$ may be empty, which is
equivalent to the impossibility to compute the resolvents defining $u^*$
and $v^*$.)
\item \label{evening5}
 $S^*=(R-\Id)(S)$. Consequently, for every $(x,y)\in S$, $y-x=u^*$, i.e., the gap vector is unique.

\item\label{EandF} $E=(A_{1}/\lambda_{2})^{-1}(u^*)\cap
\left(\Yosida{1}{\big(A_{2}/\lambda_{1}\big)}\right)^{-1}(v^*)$ and
$F=\left(\Yosida{1}{\big(A_{1}/\lambda_{2}\big)}\right)^{-1}(u^*)\cap (A_{2}/\lambda_{1})^{-1}(v^*).$
\item\label{evening6}
 $J_{A_{2}/\lambda_{1}}|_{E}:E\to F: x\mapsto x+u^*$ is a bijection with inverse mapping
$J_{A_{1}/\lambda_{2}}:F\to E: y\mapsto y+v^*$.
\item \label{evening7}
\begin{align}
S &=(E\times F)\cap
(R-\Id)^{-1}(u^*,v^*)\label{schoolbreak1}\\
&=\left(\frac{A_{1}}{\lambda_{2}}\times\frac{A_{2}}{\lambda_{1}}\right)^{-1}(u^*,v^*)
    \cap (R-\Id)^{-1}(u^*,v^*).\label{schoolbreak2}
\end{align}
\item \label{evening600}
$$\Fix J_{A}=\bigg(\frac{\Yosida{1}{A_{1}}}{\lambda_{2}}\bigg)^{-1}(u^*)\bigcap \bigg(\frac{\Yosida{1}{A_{2}}}{\lambda_{1}}\bigg)^{-1}(v^*).$$
\item \label{evening80}
The sets $\Fix (J_{A}), E,F, S$ are closed and convex.
\end{enumerate}
\end{theorem}
\begin{proof}
(i). $z\in \left(\lambda_{1}\Yosida{1}{A_{1}}
+\lambda_{2}\Yosida{1}{A_{2}}\right)^{-1}(0)$ if and only if
\begin{align*}
0& = \big(\lambda_{1}\Yosida{1}{A_{1}} +\lambda_{2}\Yosida{1}{A_{2}}\big)(z)=\lambda_{1}(\Id-J_{A_{1}})(z)+\lambda_{2}(\Id-J_{A_{2}})(z)\\
&= z-(\lambda_{1}J_{A_{1}}+\lambda_{2}J_{A_{2}})(z)=z-J_{A}(z).
\end{align*}

(ii)-- equation \eqref{schoolbreak1} of (ix) follow by applying Fact~\ref{key1} with $A=A_{1}/\lambda_{2}$, $B=A_{2}/\lambda_{1}$ and $\gamma=1$. 
To show \eqref{schoolbreak2}, we assuem that $S$ and $S^*$ are nonempty.
Note that
$$S=\left\{(x,y): \ (0,0)\in \left( (\Id -R)+\bigg(\frac{A_{1}}{\lambda_{2}}\times
\frac{A_{2}}{\lambda_{1}}\bigg)\right)(x,y)\right\}$$
and that
$$S^*=\left((\Id-R)^{-1}+\bigg(\frac{A_{1}}{\lambda_{2}}
\times\frac{A_{2}}{\lambda_{1}}\bigg)^{-1}\right)^{-1}(0,0),$$
so we can use Fact~\ref{dualityat}(i) to get
$$S=\{(x,y):\ (\exists (u^*,v^*)\in S^*)\ (u^*,v^*)\in \bigg(\frac{A_{1}}{\lambda_{2}}\times
\frac{A_{2}}{\lambda_{1}}\bigg)(x,y), -(u^*,v^*)\in (\Id-R)(x,y)\}.$$
By \ref{evening4}, $S^*$ is singleton so that $S^*=\{(u^*,v^*)\}$. 
Hence \eqref{schoolbreak2} holds.

(x). 
By \ref{set1}, 
$z\in\Fix J_{A}\Leftrightarrow 0\in \bigg(\frac{\Yosida{1}{A_{1}}}{\lambda_{2}}+\frac{\Yosida{1}{A_{2}}}{\lambda_{1}}\bigg)(z)$.
The latter has its Attouch-Thera dual given by
$$0\in \bigg(\frac{\Yosida{1}{A_{1}}}{\lambda_{2}}\bigg)^{-1}(x^*)+\widetilde{\bigg(\frac{\Yosida{1}{A_{2}}}{\lambda_{2}}
\bigg)}(x^*),$$
equivalently by \eqref{lastday} (with $\gamma=1$)
$$0\in [\Id +(A_{1}/\lambda_{2})^{-1}+\widetilde{A_{2}/\lambda_{1}}](z^*),$$
and it has a unique solution $u^*$ by (v). Fact~\ref{dualityat}(i) gives that
$z\in \Fix J_{A}$ if and only if
$$u^*\in \bigg(\frac{\Yosida{1}{A_{1}}}{\lambda_{2}}\bigg)(z),\quad  -u^*=v^*\in \bigg(\frac{\Yosida{1}{A_{2}}}{\lambda_{1}}\bigg)(z), \text{ i.e.,}$$
$$z\in \bigg(\frac{\Yosida{1}{A_{1}}}{\lambda_{2}}\bigg)^{-1}(u^*)\bigcap \bigg(\frac{\Yosida{1}{A_{2}}}{\lambda_{1}}\bigg)^{-1}(v^*).$$

(xi). It is well-known that if $B:\HH\To 2^{\HH}$ is maximal monotone, then $B(x)$ is closed
and convex for every $x\in\HH$. Observe that
$\lambda_{1}\Yosida{1}{A_{1}} +\lambda_{2}\Yosida{1}{A_{2}}$,
$\frac{A_{1}}{\lambda_{2}}+
\Yosida{1}{\big(\tfrac{A_{2}}{\lambda_{1}}\big)}$,
$\Yosida{1}{\big(\tfrac{A_{1}}{\lambda_{2}}\big)}+
\frac{A_{2}}{\lambda_{1}}$, and
$(\Id -R)+\big(\frac{A_{1}}{\lambda_{2}}\times
\frac{A_{2}}{\lambda_{1}}\big)$
are maximal monotone operators by Rockafellar's sum theorem, see \cite[pages 104--105]{Simons} or
\cite{rockmax}.
Then $\big(\lambda_{1}\Yosida{1}{A_{1}} +\lambda_{2}\Yosida{1}{A_{2}}\big)^{-1}$,
$\big(\frac{A_{1}}{\lambda_{2}}+
\Yosida{1}{\big(\tfrac{A_{2}}{\lambda_{1}}\big)}\big)^{-1}$, $\big(\Yosida{1}{\big(\tfrac{A_{1}}{\lambda_{2}}\big)}+
\frac{A_{2}}{\lambda_{1}}\big)^{-1}$, $\big((\Id -R)+\big(\frac{A_{1}}{\lambda_{2}}\times
\frac{A_{2}}{\lambda_{1}}\big)\big)^{-1}$ are maximal monotone. Hence the result 
holds by the definitions of these sets given in \ref{set1}--\ref{set4}.
\end{proof}

\subsection{Relationship among solution sets}\label{keyresultnow}
Recall that
\begin{equation}\label{usedlater}
J_{A}=\lambda_{1} J_{A_{1}}+\lambda_{2}J_{A_{2}}\quad \text{
 with $\lambda_{1}+\lambda_{2}=1$ and $\lambda_{i}>0$.}
 \end{equation}
We now study the relationships among
$\Fix J_{A},$ $ E=\Fix(J_{A_{1}/\lambda_{2}}J_{A_{2}/\lambda_{1}}),$
$F=\Fix(J_{A_{2}/\lambda_{1}}J_{A_{1}/\lambda_{2}}), \text{ and }$
$$S=\menge{(x,y)}{x=J_{A_{1}/\lambda_{2}}y, y=J_{A_{2}/\lambda_{1}}x}.$$

\begin{lemma}\label{general}
\begin{enumerate}
\item
If
$x=J_{A_{1}/\lambda_{2}}y,\quad y=J_{A_{2}/\lambda_{1}}x,$
then
$$\lambda_{1}x+\lambda_{2}y\in\Fix J_{A}, \quad x\in\Fix J_{A_{1}/\lambda_{2}}J_{A_{2}/\lambda_{1}},
\quad y\in \Fix J_{A_{2}/\lambda_{1}}J_{A_{1}/\lambda_{2}}.$$

\item\label{com1} If $x=J_{A_{1}/\lambda_{2}}J_{A_{2}/\lambda_{1}}x$, put $y=J_{A_{2}/\lambda_{1}}x$, then
$\lambda_{1}x+\lambda_{2}y\in \Fix J_{A}$.
\item\label{com2} If $y=J_{A_{2}/\lambda_{1}}J_{A_{1}/\lambda_{2}}y$, put $x=J_{A_{1}/\lambda_{2}}y$, then
$\lambda_{1}x+\lambda_{2}y\in \Fix J_{A}$.
\end{enumerate}
\end{lemma}
\begin{proof}
(i). We have
$$x\in J_{A_{1}/\lambda_{2}}y\quad \Leftrightarrow \quad y\in \frac{A_{1}}{\lambda_{2}}x+x,$$
$$y\in J_{A_{2}/\lambda_{1}}x\quad \Leftrightarrow \quad x\in \frac{A_{2}}{\lambda_{1}}y+y,$$
so that
$-\lambda_{2}x+\lambda_{2}y\in A_{1}x$ and $\lambda_{1}x-\lambda_{1}y\in A_{2}y$. Then
$\lambda_{1}x+\lambda_{2}y\in A_{1}x+x$, $\lambda_{1}x+\lambda_{2}y\in A_{2}y+y$, equivalently
$x=J_{A_{1}}(\lambda_{1}x+\lambda_{2}y), y=J_{A_{2}}(\lambda_{1}x+\lambda_{2}y)$. This gives
$$\lambda_{1}x+\lambda_{2}y=\lambda_{1}J_{A_{1}}(\lambda_{1}x+\lambda_{2}y)+\lambda_{2}J_{A_{2}}
(\lambda_{1}x+\lambda_{2}y)=[\lambda_{1}J_{A_{1}}+\lambda_{2}J_{A_{2}}
](\lambda_{1}x+\lambda_{2}y).$$
Hence $\lambda_{1}x+\lambda_{2}y\in \Fix J_{A}$.

(ii) and (iii): In either \ref{com1} or \ref{com2}, we have
$x=J_{A_{1}/\lambda_{2}}y,\quad y=J_{A_{2}/\lambda_{1}}x.$
Hence (i) applies.
\end{proof}
\begin{lemma}\label{decomposition}
If $x\in \Fix J_{A}$, then
\begin{equation}\label{avatar1}
J_{A_{1}}x=J_{A_{1}/\lambda_{2}}(J_{A_{2}}x),
\end{equation}
\begin{equation}\label{avatar2}
J_{A_{2}}x=J_{A_{2}/\lambda_{1}}(J_{A_{1}}x).
\end{equation}
Consequently,
$J_{A_{1}}x\in \Fix J_{A_{1}/\lambda_{2}}J_{A_{2}/\lambda_{1}}$ and
$J_{A_{2}}x\in \Fix J_{A_{2}/\lambda_{1}}J_{A_{1}/\lambda_{2}}.$
\end{lemma}
\begin{proof} Let us show \eqref{avatar1}.
By assumption,
 $x=\lambda_{1}J_{A_{1}}x+\lambda_{2}J_{A_{2}}x,$
 we write
\begin{equation}\label{fixpoint}
J_{A_{2}}x=\frac{x-\lambda_{1}J_{A_{1}}x}{\lambda_{2}}.
\end{equation}
We have
\begin{align*}
J_{A_{1}}x=J_{A_{1}}x
\Leftrightarrow &\quad  x\in (A_{1}+\Id)(J_{A_{1}}x)=A_{1}(J_{A_{1}}x)+J_{A_{1}}x\\
\Leftrightarrow &\quad x\in A_{1}(J_{A_{1}}x)+\lambda_{1}J_{A_{1}}x +\lambda_{2}J_{A_{1}}x \quad \text{(since
$\lambda_{1}+\lambda_{2}=1$)}\\
\Leftrightarrow & \quad x-\lambda_{1}J_{A_{1}}x\in A_{1}(J_{A_{1}}x)+\lambda_{2}J_{A_{1}}x=(A_{1}+\lambda_{2}\Id)(J_{A_{1}}x)\\
\Leftrightarrow
&
\quad \frac{x-\lambda_{1}J_{A_{1}}x}{\lambda_{2}}\in\big(A_{1}/\lambda_{2}+\Id\big)(J_{A_{1}}x)\\
\Leftrightarrow
& \quad J_{A_{2}}x\in \big(A_{1}/\lambda_{2}+\Id\big)(J_{A_{1}}x) \quad \text{ (by \eqref{fixpoint})}\\
\Leftrightarrow
&\quad  J_{A_{1}}x=J_{A_{1}/\lambda_{2}}(J_{A_{2}}x).
\end{align*}
The proof of \eqref{avatar2} is similar.
\end{proof}

Note that
$$S=\menge{(x,y)}{x=J_{A_{1}/\lambda_{2}}y, y=J_{A_{2}/\lambda_{1}}x}.$$
\begin{theorem}\label{StoJ}
 Define
$$T:S\to \Fix J_{A}: (x,y)\mapsto \lambda_{1}x+\lambda_{2}y.$$
Then $T$ is a homeomorphism, and the inverse of $T$ is given by
$$T^{-1}:\Fix J_{A} \to S: z\mapsto (J_{A_{1}}z,J_{A_{2}}z).$$
Consequently, $\Fix J_{A}=L(S)$ where $L:\HH\times \HH\rightarrow \HH: (x,y)\mapsto \lambda_1 x+\lambda_{2}y$.
\end{theorem}
\begin{proof}
For every $(x,y)\in S$, by Lemma~\ref{general}(i), $T(x,y)\in \Fix J_{A}$,
so $T(S)\subseteq \Fix J_{A}$.
For every $z\in \Fix J_{A}$, by Lemma~\ref{decomposition},
$(J_{A_{1}}z, J_{A_{2}}z)\in S$ and $z=\lambda_{1}J_{A_{1}}(z)+\lambda_{2}J_{A_{2}}(z)=T(J_{A_{1}}z, J_{A_{2}}z)$, thus $T(S)\supseteq \Fix J_{A}$. Hence $T(S)=\Fix J_{A}$, i.e., $T$ is onto.
To show that $T$ is one-to-one, let $(x_{i},y_{i})\in S$
for $i\in\{1,2\}$.
If $T(x_{1},y_{1})=T(x_{2},y_{2})$, i.e., $\lambda_{1}x_{1}+\lambda_{2}y_{1}=\lambda_{1}x_{2}+\lambda_{2}y_{2}$, then
$$\lambda_{1}(x_{1}-y_{1})+y_{1}=\lambda_{1}(x_{2}-y_{2})+y_{2}.$$
By Theorem~\ref{evening}\ref{evening4} and \ref{evening5}, $S^*$ is unique and $x_{1}-y_{1}=x_{2}-y_2=v^*$, thus $y_{1}=y_{2}$ and $x_{1}=x_{2}$.

Since for $z\in \Fix J_{A}$, $(J_{A_{1}}(z),J_{A_{2}}(z))\in S$ and $z=T(J_{A_{1}}(z),J_{A_{2}}(z))$,
$T$ is one-to-one and onto, we obtain that
$T^{-1}(z)=(J_{A_{1}}(z),J_{A_{2}}(z)).$ In addition, both $T, T^{-1}$ are continuous. Hence
$T: S\to \Fix J_{A}$ is a homeomorphism.
\end{proof}
\begin{theorem}\label{monday}
\begin{enumerate}
\item\label{EtoS} The mapping
$$T_{1}:E\to S: x\mapsto (x, J_{A_{2}/\lambda_{1}}x),$$
is a homeomorphism and its inverse is given by
$$T_{1}^{-1}:S\to E: (x,y)\mapsto x.$$
\item\label{FtoS} The mapping
$$T_{2}:F\to S: y\mapsto (J_{A_{1}/\lambda_{2}}y,y),$$
is a homeomorphism and its inverse is given by
$$T_{2}^{-1}:S\to F: (x,y)\mapsto y.$$
\end{enumerate}
\end{theorem}
\begin{proof}
We only prove \ref{EtoS}, since \ref{FtoS} can be proved similarly.
To see \ref{EtoS}, let $x\in E$. By the definition of $E$,
$x=J_{A_{1}/\lambda_{2}}J_{A_{2}/\lambda_{1}}x$. Put $y=J_{A_{2}/\lambda_{1}}x$.
We have
$$x=J_{A_{1}/\lambda_{2}}y, \quad y=J_{A_{2}/\lambda_{1}}x,$$
whence $T_{1}x=(x,y)\in S$. Therefore, $T_{1}(E)\subseteq S$. Now for every $(x,y)\in S$,
by the definition of $S$,
$$x=J_{A_{1}/\lambda_{2}}y, \quad y=J_{A_{2}/\lambda_{1}}x,$$
then $x\in E$ and $(x,y)=(x, J_{A_{2}/\lambda_{1}}x)=T_{1}x.$ Therefore, $S\subseteq T_{1}(E)$. Hence
$T_{1}(E)=S$. Clearly, $T_{1}$ is one-to-one.
Altogether, $T_{1}$ is one-to-one and onto. Since for every $(x,y)\in S$,
$$(x,y)=(x, J_{A_{2}/\lambda_{1}}x)=T_{1}x,$$
we have $T_{1}^{-1}(x,y)=x$.
\end{proof}

The next result provides a partial answer to a question
raised by C.~Byrne (see \cite[page~305]{charles}).
It provides the transformations to go back and forth
between fixed point sets of compositions of resolvents and
the fixed point set of the average.

\begin{theorem}\label{plumber}
Let $u^*$ be given as in Theorem~\ref{evening}(v).
\begin{enumerate}
\item\label{monday1}
The mapping $$ H_{1}: E\to \Fix J_{A}:
x\mapsto \lambda_{1} x+\lambda_{2} J_{A_{2}/\lambda_{1}}x
=x+\lambda_{2} u^*,$$
is a homeomorphism. Moreover, $H_{1}^{-1}:\Fix J_{A}\to E$ is given by $H_{1}^{-1}(z)=
J_{A_{1}}(z)$. Hence
\begin{equation}\label{break1}
\Fix J_{A}=E+\lambda_{2}u^*.
\end{equation}
\item \label{monday2}
The mapping
$$H_{2}: F\to\Fix J_{A}: y\mapsto \lambda_{1}J_{A_{1}/\lambda_{2}}y+\lambda_{2}y=
-\lambda_1 u^*+y,$$
is a homeomorphism. Moreover, $H_{2}^{-1}:\Fix J_{A}\to F$ is given by $H_{2}^{-1}(z)=
J_{A_{2}}(z)$. Hence
\begin{equation}\label{break2}
\Fix J_{A}=F-\lambda_{1}u^*.
\end{equation}
\end{enumerate}
\end{theorem}
\begin{proof} Combine Theorem~\ref{StoJ} and Theorem~\ref{monday}. Using the same notations as in
Theorem~\ref{StoJ} and Theorem~\ref{monday}, \ref{monday1} follows from $H_{1}=T\circ T_{1}$;
\ref{monday2} follows from $H_{2}=T\circ T_{2}$. Moreover,
$$(\forall \ x\in E)\ H_{1}(x)=\lambda_{1} x+\lambda_{2} J_{A_{2}/\lambda_{1}}x=x+\lambda_{2}(J_{A_{2}/\lambda_{1}}x-x)=x+\lambda_{2}u^*,$$
$$(\forall \ y\in F)\ H_{2}(y)=\lambda_{1}J_{A_{1}/\lambda_{2}}y+\lambda_{2}y=\lambda_{1}(J_{A_{1}/\lambda_{2}}y-y)+y
=\lambda_{1}(-u^*)+y,$$
by Theorem~\ref{evening}\ref{evening5}. Hence \eqref{break1} and \eqref{break2} hold.
\end{proof}

\begin{corollary}\label{allsets}
The following is true.
\begin{enumerate}
\item $E\neq\varnothing$ $\Leftrightarrow$
$F\neq\varnothing$ $\Leftrightarrow$
$S\neq\varnothing$ $\Leftrightarrow$
$S^*\neq\varnothing$ $\Leftrightarrow$
$\Fix J_{A}\neq\varnothing$.
\item
$E$ is a singleton $\Leftrightarrow$
$F$ is a singleton $\Leftrightarrow$
$S$ is a singleton $\Leftrightarrow$
$\Fix J_{A}$ is a singleton.
\item $\Fix J_{A}= \lambda_{1}E+\lambda_{2}F.$
\end{enumerate}
\end{corollary}

\begin{proof} (i) and (ii) follow from Theorems \ref{StoJ}, \ref{monday} and Theorem~\ref{plumber}.
It remains to prove (iii). By Theorem~\ref{plumber},
$$\Fix J_{A}=E+\lambda_{2}u^*, \quad \Fix J_{A}=F-\lambda_{1}u^*.$$
As $\Fix J_{A}$ is convex by Theorem~\ref{evening}\ref{evening80}, we obtain
$$\Fix J_{A}=\lambda_{1}\Fix J_{A}+\lambda_{2}\Fix J_{A}=\lambda_{1}(E+\lambda_{2}u^*)+\lambda_{2}
(F-\lambda_{1}u^*)= \lambda_{1}E+\lambda_{2} F,$$
as claimed.
\end{proof}

\subsection{The case when $\Fix J_{A_{1}/\lambda_2}\cap \Fix J_{A_{2}/\lambda_{1}}\neq\varnothing$}
Note that
\begin{equation}
\label{thesameset}
\Fix J_{A_{1}}=\Fix J_{A_{1}/\lambda_{2}},\quad  \Fix J_{A_{2}}
=\Fix J_{A_{2}/\lambda_{1}}.
\end{equation}

\begin{theorem}
\label{t:thesameset}
Assume that $\Fix J_{A_{1}/\lambda_{2}}\cap \Fix J_{A_{2}/\lambda_{1}}\neq\varnothing$.
Let $\lambda_{1}+\lambda_{2}=1$ with $\lambda_{i}> 0$. Then
\begin{align*}
\Fix(\lambda_1 J_{A_{1}}+\lambda_{2}J_{A_{2}}) & =\Fix J_{A_{1}}\cap \Fix J_{A_{2}}=  \Fix J_{A_{1}/\lambda_{2}}\cap \Fix J_{A_{2}/\lambda_{1}}\\
    & =\Fix(J_{A_{1}/\lambda_{2}}\circ J_{A_{2}/\lambda_1})=\Fix(J_{A_{2}/\lambda_{1}}\circ J_{A_{1}/\lambda_2}).
\end{align*}
\end{theorem}
\begin{proof}
In view of \eqref{thesameset},
$\Fix J_{A_{1}}\cap \Fix J_{A_{2}} \neq\varnothing.$
Since every resolvent is attracting, it suffices to apply Fact~\ref{onesetonly}.
\end{proof}

\section{Minimizers of the proximal average}\label{proxcomp}

We now specialize our results to
$A_{1}=\partial f_{1}$ and $A_{2}=\partial f_{2}$ for
two proper lower semicontinuous convex functions
$f_{1},f_{2}$.
This allows us to understand the results
of Section~\ref{resolvcomp} from the variational analysis perspective.
Let $f_{1},f_{2}\in\GX(\HH)$ and $\lambda_{1}+\lambda_{2}=1$
with each $\lambda_{i}>0$.

\subsection{Minimization problem formulations and their common Fenchel-Rockafellar
dual}
We consider
\begin{equation}\label{mixed}
\min_{(x,y)}\  g(x,y):=\bigg(\frac{f_{1}(x)}{\lambda_{2}}+\frac{f_{2}(y)}{\lambda_{1}}+\frac{\|x-y\|^2}{2}\bigg).
\end{equation}
This turns out to be closely related to the
proximal average of $f_{1},f_{2}$,
recently studied in \cite{Baus08,lucet,Baus09,moffat}.
Recall
\begin{equation}\label{ura1}
(\forall z\in X)\quad \pflm(z)=
\inf_{z=\lambda_{1}x+\lambda_{2}y }\big(
\lambda_{1}f_1(x)+\lambda_{2}f_{2}(y)+
\tfrac{\lambda_{1}\lambda_{2}}{2\gamma}\|x-y\|^2\big),
\end{equation}
where $ \fettf = (f_1,f_{2}), \fettla = (\lambda_1,\lambda_2), \gamma>0$.
When $\gamma=1$, we just write $\pfl$.
Therefore, \eqref{mixed} has the same minimum value as the scaled proximal average
\begin{equation}\label{ura2}
\min_{z} \frac{\pfl(z)}{\lambda_{1}\lambda_{2}}.
\end{equation}
In terms of Moreau envelopes,
\eqref{mixed} can be reformulated as
\begin{equation}\label{meeting1}
\min_{x}\ g_{1}(x):=f_{1}(x)/\lambda_{2}+e_{1}(f_{2}/\lambda_{1})(x)
\end{equation}
and
\begin{equation}\label{meeting2}
\min_{y}\ g_{2}(y):=e_{1}(f_{1}/\lambda_{2})(y)+f_{2}(y)/\lambda_{1}.
\end{equation}
With regard to \eqref{ura2}, we also consider
\begin{equation}\label{yong}
\min_{z}\ g_{3}(z):=\lambda_{1}e_{1}f_{1}(z)+\lambda_{2}e_{1}f_{2}(z)=\min_{z} \ \lambda_{1}\lambda_{2}
\bigg(\frac{e_{1}f_{1}(z)}{\lambda_{2}}+\frac{e_{1}f_{2}(z)}{\lambda_{1}}\bigg).
\end{equation}

The following facts about proximal average will be useful.

\begin{fact}\label{proxpaper}
\begin{enumerate}
\item \emph{(See \cite[Theorem 4.10]{Baus08}.)}
For every $z\in\dom\pflm$,
there exist $x\in\dom f_{1}, y\in \dom f_{2}$ such that
$z=\lambda_{1}x+\lambda_{2}y$ and $$\pflm(z) =
\lambda_{1}f_{1}(x)+\lambda_{2}f_{2}(y)+
\frac{\lambda_{1}\lambda_{2}\|x-y\|^2}{2\gamma}.$$
\item \emph{(See \cite[Theorem 6.2]{Baus08}.)}
$e_{\gamma}\pflm = \lambda_1 e_{\gamma} f_1 + \lambda_2 e_\gamma f_2$.
\item \emph{(See \cite[Theorem 6.7]{Baus08}.)}
$\prox_{\pfl}=\lambda_{1}\prox_{f_{1}}+\lambda_{2}\prox_{f_{2}}$.
\end{enumerate}
\end{fact}

\begin{fact}
\label{envelope}
Let $f\in\GX(\HH)$ and $\gamma\in (0,+\infty)$. Then
\begin{enumerate}
\item $e_{\gamma}f$ is Fr\'echet differentiable on $\HH$ and
$\nabla (e_{\gamma}f)
=\Yosida{\gamma}{(\partial f)}
=(\Id-\prox_{\gamma f})/\gamma$.
\item $\inf (e_{\gamma}f)=\inf f$
and $\argmin (e_{\gamma}f)=\argmin f$.
\end{enumerate}
\end{fact}
\begin{proof}
See \cite[Lemma~2.5]{patrick2} and \cite{moreau}.
See also \cite[Example~1.46 and Theorem~2.26]{Rock98}.
\end{proof}

\begin{proposition}\label{onlyoneneeded}
For the minimization problems given by
\eqref{mixed}--\eqref{yong}, we have
$$\min_{(x,y)} g(x,y)=\min_{z}\frac{\pfl(z)}{\lambda_{1}\lambda_{2}}
=\min_{x}g_{1}(x)=\min_y g_{2}(y)
=\min_{z}\frac{g_{3}(z)}{\lambda_{1}\lambda_{2}}.$$
\end{proposition}
\begin{proof} While the first three equalities are immediate,
the fourth one follows from Fact~\ref{proxpaper}(ii)
and Fact~\ref{envelope}(ii).
\end{proof}

The following result is a convex-function refinement
of Theorem~\ref{operator}. It says that
convex optimization problems~\eqref{mixed}, \eqref{meeting1}, 
\eqref{meeting2}, \eqref{yong} 
share one common Fenchel-Rockafellar dual problem.

\begin{theorem}\label{solongso} Up to a `$-$' sign change in
the dual variable, the following problems have
the same Fenchel dual.
\begin{enumerate}
\item
$$
\min_{(x,y)}\ \bigg(\frac{f_{1}(x)}{\lambda_{2}}+\frac{f_{2}(y)}{\lambda_{1}}+\frac{\|x-y\|^2}{2}\bigg).
$$
\item
$$
\min_{x}\ f_{1}(x)/\lambda_{2}+e_{1}(f_{2}/\lambda_{1})(x).
$$
\item
$$
\min_{y}\ e_{1}(f_{1}/\lambda_{2})(y)+f_{2}(y)/\lambda_{1}.
$$

\item
$$
\min_{z}\ \frac{e_{1}f_{1}(z)}{\lambda_{2}}+\frac{e_{1}f_{2}(z)}{\lambda_{1}}.
$$
\end{enumerate}
Namely, up to a `$-$' sign change in the dual variable,
their Fenchel-Rockafellar
dual is given by
\begin{align}
(D)\quad &\max_{\phi}\left[-\bigg(\frac{f_{1}}{\lambda_{2}}\bigg)^*(-\phi)-\bigg(\frac{f_{2}}{\lambda_{1}}\bigg)^*(\phi)-
\frac{\|\phi\|^{2}}{2}\right] \label{dualprob}\\
&=-\min_{\phi}\left[\bigg(\frac{f_{1}}{\lambda_{2}}\bigg)^*(-\phi)+
\bigg(\frac{f_{2}}{\lambda_{1}}\bigg)^*(\phi)+
\frac{\|\phi\|^{2}}{2}\right].\nonumber
\end{align}
\end{theorem}

\begin{proof}
(i).
Using the Fenchel-Rockafellar Duality theorem
(Fact~\ref{duality}) for
$f_{1}/\lambda_{2} \oplus f_{2}/\lambda_{1}$,
$j\circ L$ with $j=\|\cdot\|^2/2$ and $L=(\Id, -\Id):\HH\times\HH\to\HH$,
we obtain the dual problem \eqref{dualprob}.

(ii). The Fenchel dual is given by
$$\sup_{\phi}-(f_{1}/\lambda_{2})^{*}(-\phi)-[e_{1}(f_{2}/\lambda_{1})]^{*}(\phi).$$
As
$[e_{1}(f_{2}/\lambda_{1})]^{*}=(f_{2}/\lambda_{1})^*+j,$ its Fenchel dual becomes
$$\sup_{\phi}-(f_{1}/\lambda_{2})^{*}(-\phi)-(f_{2}/\lambda_{1})^*(\phi)-\|\phi\|^2/2.$$
(iii). The Fenchel dual is
\begin{align*}
&\sup_{\phi}-[e_{1}(f_{1}/\lambda_{2})]^*(-\phi)-(f_{2}/\lambda_{1})^*(\phi)\\
&=\sup_{\phi}-(f_{1}/\lambda_{2})^*(-\phi)-\|-\phi\|^2/2-(f_{2}/\lambda_{1})^*(\phi)\\
&=\sup_{\phi}-(f_{1}/\lambda_{2})^*(-\phi)-(f_{2}/\lambda_{1})^*(\phi)-\|\phi\|^2/2.
\end{align*}
(iv). Its Fenchel dual is
\begin{equation}\label{paint}
\sup_{\phi}-\bigg(\frac{e_{1}f_{1}}{\lambda_{2}}\bigg)^*(-\phi)-
\bigg(\frac{e_{1}f_{2}}{\lambda_{1}}\bigg)^*(\phi).
\end{equation}
Now
\begin{align*}
\bigg(\frac{e_{1}f_{1}}{\lambda_{2}}\bigg)^* &=\frac{1}{\lambda_{2}}(e_{1}f_{1})^*(\lambda_{2}\cdot)
=\frac{1}{\lambda_{2}}(f_{1}^*+j)(\lambda_{2}\cdot)\\
&=\frac{1}{\lambda_{2}}f_{1}^*(\lambda_{2}\cdot)+\lambda_{2}j=\bigg(\frac{f_{1}}{\lambda_{2}}\bigg)^*
+\lambda_{2}j.
\end{align*}
Similarly,
$$\bigg(\frac{e_{1}f_{2}}{\lambda_{1}}\bigg)^*=\bigg(\frac{f_{2}}{\lambda_{1}}\bigg)^*+\lambda_{1}j.$$
Then \eqref{paint} becomes
\begin{align*}
& \sup_{\phi}-\bigg(\frac{f_{1}}{\lambda_{2}}\bigg)^*(-\phi)
-\lambda_{2}j(-\phi)-\bigg(\frac{f_{2}}{\lambda_{1}}\bigg)^*(\phi)-\lambda_{1}j(\phi)\\
&=\sup_{\phi}-\bigg(\frac{f_{1}}{\lambda_{2}}\bigg)^*(-\phi)-
\bigg(\frac{f_{2}}{\lambda_{1}}\bigg)^*(\phi)-\frac{\|\phi\|^2}{2}.
\end{align*}
The proof is complete.
\end{proof}

When the primal problem \eqref{mixed} has a finite infimum value,
the primal optimal value and the dual
optimal value are equal; moreover,
the dual optimal value is attained.

While the solution set of the primal problem \eqref{mixed}
(see Theorem~\ref{proxfunctions}(i))
may be empty, the solution set of the dual problem
\eqref{dualprob}
is nonempty and a singleton as long as \eqref{mixed}
has a finite infimum value.
This feature of Fenchel-Rockafellar duality is in stark contrast to the
Attouch-Th\'era duality of Section~\ref{resolvcomp}; see
Corollary~\ref{allsets}(i).

\begin{theorem} When the primal problem \eqref{mixed} has a finite infimum, the
dual (D) has a unique solution
\begin{equation}\label{cape1}
\bar{\phi} =\prox_{(f_{1}/\lambda_{2})^*\circ (-\Id)+(f_{2}/\lambda_{1})^{*}}(0).
\end{equation}
If
\begin{equation}\label{cq}
\dom (f_{1}/\lambda_{2})^*\cap -\inte \dom (f_{2}/\lambda_{2})^*\neq \varnothing \quad \text{or}\quad
\inte \dom (f_{1}/\lambda_{2})^*\cap -\dom (f_{2}/\lambda_{2})^*\neq\emp,
\end{equation}
 then
\begin{equation}\label{cape2}
\bar{\phi}=J_{\widetilde{\partial f_{1}/\lambda_{2}}+
(\partial f_{2}/\lambda_{1})^{-1}}(0)=-J_{(\partial f_{1}/\lambda_{2})^{-1}+
\widetilde{\partial f_{2}/\lambda_{1}}}(0).
\end{equation}
\end{theorem}
\begin{proof}
From \eqref{dualprob}, we have
$$
0\in \partial \big[(f_{1}/\lambda_{2})^*\circ (-\Id)+(f_{2}/\lambda_{1})^{*}\big](\bar{\phi})
+\bar{\phi},$$
so \eqref{cape1} holds.

Under the assumption \eqref{cq}, we can apply the chain rule so that
$$0\in -(\partial f_{1}/\lambda_{2})^{-1}(-\bar{\phi})+(\partial f_{2}/\lambda_{1})^{-1}(\bar{\phi})
+\bar{\phi}
=\widetilde{\partial f_{1}/\lambda_{2}}(\bar{\phi})+(\partial f_{2}/\lambda_{1})^{-1}(\bar{\phi})+\bar{\phi}.$$
Hence the first equality in \eqref{cape2} holds. Rewrite the dual problem \eqref{dualprob}
as
$$-\inf_{\psi}\left[\bigg(\frac{f_{1}}{\lambda_{2}}\bigg)^*(\psi)+
\bigg(\frac{f_{2}}{\lambda_{1}}\bigg)^*(-\psi)+
\frac{\|\psi\|^{2}}{2}\right]$$
and denote its optimal solution by $\bar{\psi}$. Then $-\bar{\psi}=\bar{\phi}$ and
$\bar{\psi}=J_{(\partial f_{1}/\lambda_{2})^{-1}+
\widetilde{\partial f_{2}/\lambda_{1}}}(0)$. Therefore, the second equality in \eqref{cape2} holds
also.
\end{proof}

\begin{remark}
\label{r:blablabla}
Note that \cite[Proposition~4.3]{Baus05} also implies \eqref{cape1}
as well as 
\begin{equation*}
\bar{\phi} = -\prox_{\left[ (f_2/\lambda_1)\Box
(f_1/\lambda_2\circ(-\Id))\right]^{**}}(0).
\end{equation*}
Observe that $\bar{\phi}=v^*$ as given in Theorem~\ref{evening}(v).
\end{remark}

\subsection{Characterization of minimizers}

Set
\begin{align*}
S &:=\menge{(x,y)}{x=\prox_{f_{1}/\lambda_{2}}y, \
y=\prox_{f_{2}/\lambda_{1}}x},\\
E&:=\Fix (\prox_{f_{1}/\lambda_{2}}\prox_{f_{2}/\lambda_{1}}),\\
F &:=\Fix(\prox_{f_{2}/\lambda_{1}}\prox_{f_{1}/\lambda_{2}}).
\end{align*}
\begin{theorem}\label{proxfunctions}
The following assertions hold.
\begin{enumerate}
\item \emph{({\bf Fixed points of alternating proximal mappings})}
$S=\argmin g\subseteq E\times F$.
\item \emph{({\bf Fixed points of proximal mapping composition})}
$E=\argmin g_{1}$.
\item \emph{({\bf Fixed points of proximal mapping composition})}
$F=\argmin g_{2}$.
\item\label{function1} \label{related2}
$\argmin\pfl=\argmin g_{3}=
\left[\lambda_{1}\Yosida{1}{(\partial f_{1})}+
\lambda_{2}\Yosida{1}{(\partial f_{2})}\right]^{-1}(0)$.
\item \emph{({\bf Fixed points of the average of proximal mappings})}
\label{function2}
\begin{align}
\argmin\pfl & =\menge{\lambda_{1}x+\lambda_{2}y}{(x,y)\in S}
\label{related1}\\
&=\Fix(\prox_{\pfl})=\Fix (\lambda_{1}\prox_{f_{1}}+\lambda_{2}\prox_{f_{2}}).\label{related3}
\end{align}
\item The sets $S,E,F,\Fix (\lambda_{1}\prox_{f_{1}}+\lambda_{2}\prox_{f_{2}})$ are closed and convex.
\end{enumerate}
\end{theorem}

\begin{proof}
We have
\begin{align*}
\partial g(x,y) &=\bigg(\frac{\partial f_{1}(x)}{\lambda_{2}}+x-y,\frac{\partial f_{2}(y)}{\lambda_{1}}+y-x\bigg)\\
&=\bigg(\frac{\partial f_{1}}{\lambda_{2}}\times \frac{\partial f_{2}}{\lambda_{1}}+(\Id-R)\bigg)(x,y).
\end{align*}
Moreover,
using $\nabla e_{1}f_{i}=\yosida{1}(\partial f_{i})$,
$$\nabla e_{1}(f_1/\lambda_{2})=
\Yosida{1}{\left[\partial f_{1}/\lambda_{2}\right]}
=\Id-\prox_{f_{1}/\lambda_{2}}, \quad
\nabla e_{1}(f_2/\lambda_{1})
=\Yosida{1}{\left[\partial f_{2}/\lambda_{1}\right]}
=\Id-\prox_{f_{2}/\lambda_{1}},$$
 by Fact~\ref{envelope}, we obtain
\begin{align}
\partial g_1 &=\partial f_{1}/\lambda_{2}+\yosida{1}[\partial f_{2}/\lambda_{1}]=\partial f_{1}/\lambda_{2}+\Id-\prox_{f_{2}/\lambda_{1}},\label{grad1}\\
\partial g_{2} &= \partial f_{2}/\lambda_{1}+\yosida{1}[\partial f_{1}/\lambda_{2}]=\partial f_{2}/\lambda_{1} +\Id-\prox_{f_{1}/\lambda_{2}},\label{grad2}\\
\nabla g_{3}& =\lambda_{1}\yosida{1}(\partial f_{1})+\lambda_{2}\yosida{1}(\partial f_{2})=
\Id-(\lambda_{1}\prox_{f_{1}}+\lambda_{2}\prox_{f_{2}}).\label{grad3}
\end{align}
Then (i)--(iii) follows from Theorem~\ref{evening} by using $A_{1}=\partial f_{1}, A_{2}=\partial f_{2}$, or \cite[Proposition 4.1]{Baus05} by using $\gamma=1$ and functions
$f_{1}/\lambda_{2}, f_{2}/\lambda_{1}$.
To show \ref{related2}, apply Fact~\ref{proxpaper}(ii) to obtain
$e_{1}\pfl=\lambda_{1}e_{1}f_{1}+\lambda_{2}e_{1}f_{2}.$
Since $\argmin e_{1}\pfl=\argmin \pfl$ by Fact~\ref{envelope}, we have
the first equality of \ref{related2}. Furthermore,
\eqref{grad3} gives
$\argmin g_{3}= \nabla g_{3}^{-1}(0)=[\lambda_{1}\yosida{1}(\partial f_{1})+\lambda_{2}\yosida{1}(\partial f_{2})]^{-1}(0)$.

(v): We first show \eqref{related1}. Let $z\in\argmin\pfl$. By Fact~\ref{proxpaper}(i), $z=\lambda_{1}x+\lambda_{2}y$ for some $(x,y)$
with
$$\frac{\pfl(z)}{\lambda_{1}\lambda_{2}}=\frac{f_{1}(x)}{\lambda_{2}}+\frac{f_{2}(y)}{\lambda_{1}}+
\frac{\|x-y\|^{2}}{2}.$$
By Proposition~\ref{onlyoneneeded}, $\frac{\pfl(z)}{\lambda_{1}\lambda_{2}}=\min g$, so $(x,y)\in\argmin g$. As $S=\argmin g$ by (i), we have 
$z\in \menge{\lambda_{1}x+\lambda_{2}y}{(x,y)\in S}$. Conversely, if $(x,y)\in S$, then by definition of $\pfl$ and Proposition~\ref{onlyoneneeded},
$$\frac{\pfl(\lambda_{1}x+\lambda_{2}y)}{\lambda_{1}\lambda_{2}}\leq \frac{f_{1}(x)}{\lambda_{2}}+\frac{f_{2}(y)}{\lambda_{1}}+\frac{\|x-y\|^2}{2}= \min g=\min\frac{\pfl}{\lambda_{1}\lambda_{2}},$$
thus $\lambda_{1}x+\lambda_{2}y\in\argmin \pfl$. Therefore, \eqref{related1} holds.
Note that \eqref{related3} follows from Fact~\ref{proxpaper}(iii) and
Fact~\ref{elementarym}(i) for $\gamma=1$.

(vi): Indeed, these sets are $\argmin$ sets 
of lower semicontinuous convex 
functions $g, g_{1}, g_{2}, \pfl$ respectively.
\end{proof}

Problem~\eqref{yong} is a least squares problem
in terms of convex functions $f_{1},f_{2}$.
The next result is well known.

\begin{corollary}[least square solution]
Let $f_{1},f_{2}\in\GX(\HH)$ and $\lambda_1+\lambda_{2}=1$
with each $\lambda_{i}>0$.
Then
\begin{equation}
\label{least1}
\Fix(\lambda_{1}\prox_{f_{1}}+\lambda_{2}\prox_{f_{2}})=\argmin(\lambda_{1}e_{1}f_{1}+\lambda_{2}
e_{1}f_{2}).
\end{equation}
When $f_{i}=\iota_{C_{i}}$ with $C_{i}\subseteq\HH$
being nonempty closed convex, we have
\begin{equation}\label{least2}
\Fix(\lambda_{1}P_{C_{1}}+\lambda_{2}P_{C_{2}})=\argmin\bigg(\lambda_{1}\frac{d_{C_{1}}^2}{2}
+\lambda_{2}\frac{d_{C_{2}}^{2}}{2}\bigg).
\end{equation}
\end{corollary}
\begin{proof}
Combining Theorem~\ref{proxfunctions}\ref{function1} and \ref{function2} gives \eqref{least1}.
 Observe that $\prox_{\iota_{C_{i}}}=P_{C_{i}}$ and $e_{1}\iota_{C_{i}}=d_{C_{i}}^2/2$. Hence \eqref{least2} follows from \eqref{least1}.
\end{proof}

The following result says that when $S\neq\varnothing$,
for every $(x,y)\in S$ the difference $x-y$,
sometimes also called the \emph{gap vector},
is the unique solution to the dual problem.
Characterizations of $S, E, F,$
and $\Fix(\lambda_{1}\prox_{f_{1}}+\lambda_{2}\prox_{f_{2}})$
in terms of dual solution $\bar{\phi}$ come as follows.

\begin{theorem}\label{characterization}
\begin{enumerate}
\item We have
$(x,y)\in S$ and $\phi\in S^*$ if and only if
$$(-\phi,\phi)\in \partial f_{1}(x)/\lambda_{2}\times \partial f_{2}(y)/\lambda_{1},\quad
\phi=x-y.$$
\item Let $\bar{\phi}$ be the unique solution to $(D)$ and
assume that $S\neq\varnothing$.
Then for every $(x,y)\in S$, one has $x-y=\bar{\phi}$. Moreover,
$$S=\big(\partial f_{1}/\lambda_{2}\times\partial f_{2}/\lambda_{1}\big)^{-1}(-\bar{\phi},\bar{\phi})\cap (R-\Id)^{-1}(-\bar{\phi},\bar{\phi}).$$
\end{enumerate}
\end{theorem}
\begin{proof}
(i). Use
$L^*=(\Id,-\Id): \HH\rightarrow\HH\times\HH $, $f=f_{1}/\lambda_{2}\oplus f_{2}/\lambda_{1}$ and
$g=j$.
By Fact~\ref{duality} again, $(x,y)\in S$ and $\phi\in S^*$ if and only if
$$(-\phi,\phi)\in \partial f_{1}(x)/\lambda_{2}\times \partial f_{2}(y)/\lambda_{1},\quad
\phi=x-y.$$
(ii). As the dual objective function is strictly concave,
the dual solution is unique, say $\bar{\phi}$.
It suffices to apply (i).
\end{proof}

\begin{theorem}\label{penticton}
\begin{enumerate}
\item $x\in E$ if and only if
$$-\bar{\phi}\in \partial f_{1}(x)/\lambda_{2}, \quad \bar{\phi}=x-\prox_{f_{2}/\lambda_{1}}(x).$$
\item $y\in F$ if and only if
$$-\bar{\phi}=y-\prox_{f_{1}/\lambda_{2}}(y),\quad \bar{\phi}\in \partial f_{2}(y)/\lambda_{1}.$$
\item $z\in \Fix (\lambda_{1}\prox_{f_{1}}+\lambda_{2}\prox_{f_{2}})$ if and only if
$$-\bar{\phi}=\lambda_2^{-1}(z-\prox_{f_{1}}(z)),\quad
\bar{\phi}=\lambda_1^{-1}(z-\prox_{f_{2}}(z)).$$
\end{enumerate}
\end{theorem}
\begin{proof}
This follows from Theorem~\ref{solongso} and Fact~\ref{duality}.
\end{proof}

\begin{remark} Theorems~\ref{characterization} and \ref{penticton} are convex function analogues for
Theorem~\ref{evening}\ref{evening5}, \ref{EandF}, \ref{evening7} and \ref{evening600}.
\end{remark}

\subsection{Relationship among minimizers}
We now study the relationships among $\Fix(\lambda_{1}\prox_{f_{1}}+
\lambda_{2}\prox_{f_{2}})$,
$E=\Fix(\prox_{f_{1}/\lambda_{2}}\prox_{f_{2}/\lambda_{1}})$, $F=\Fix(\prox_{f_{2}/\lambda_{1}}\prox_{f_{1}/\lambda_{2}})$, and
$$S=\menge{(x,y)}{x=\prox_{f_{1}/\lambda_{2}}y, \ y=\prox_{f_{2}/\lambda_{1}}x}.$$

\begin{theorem}\label{amy1}
 Let $\bar{\phi}$ be the dual solution, i.e., the solution to \eqref{dualprob}.
Define $T_{1}:E\rightarrow
\Fix(\lambda_{1}\prox_{f_{1}}+
\lambda_{2}\prox_{f_{2}})$ by
$$T_{1}(x)=\lambda_{1}x+\lambda_{2}\prox_{f_{2}/\lambda_{1}}(x)=x-\lambda_{2}\bar{\phi}.$$
Then $T_{1}$ is a homeomorphism with $T_{1}^{-1}(z)=\prox_{f_{1}}(z).$
Consequently,
\begin{equation}\label{skate1}
\Fix(\lambda_{1}\prox_{f_{1}}+
\lambda_{2}\prox_{f_{2}})=E-\lambda_{2}\bar{\phi}.
\end{equation}
\end{theorem}
\begin{proof}
Use Theorem~\ref{plumber}\ref{monday1} with $A_{i}=\partial f_{i}$ for $i=1,2$.
\end{proof}

\begin{theorem}\label{amy2}
Let $\bar{\phi}$ be the dual solution.
Define $T_{2}:F\rightarrow
\Fix(\lambda_{1}\prox_{f_{1}}+
\lambda_{2}\prox_{f_{2}})$ by
$$T_{2}(y)=\lambda_{1}\prox_{f_{1}/\lambda_{2}}(y)+\lambda_{2}y=y+\lambda_{1}\bar{\phi}.$$
Then $T_{2}$ is a homeomorphism with $T_{2}^{-1}(z)=\prox_{f_{2}}(z).$
Consequently,
\begin{equation}\label{skate2}
\Fix(\lambda_{1}\prox_{f_{1}}+
\lambda_{2}\prox_{f_{2}})=F+\lambda_{1}\bar{\phi}.
\end{equation}
\end{theorem}
\begin{proof}
Apply Theorem~\ref{plumber}\ref{monday2} with $A_{i}=\partial f_{i}$ for $i=1,2$.
\end{proof}

\begin{theorem}\label{amy3}
Define $T: S\rightarrow \Fix(\lambda_{1}\prox_{f_{1}}+
\lambda_{2}\prox_{f_{2}})$ by
$$T(x,y)=\lambda_{1}x+\lambda_{2}y.$$
Then $T$ is a homeomorphism. Moreover, for every $z\in \Fix(\prox_{\pfl})$ one has
$T^{-1}(z)=(\prox_{f_{1}}(z),\prox_{f_{2}}(z)).$  Consequently,
$\Fix(\lambda_{1}\prox_{f_{1}}+
\lambda_{2}\prox_{f_{2}})=L(S)$ where $L:\HH\times\HH\rightarrow \HH: (x,y)\mapsto \lambda_{1}x+
\lambda_{2}y$.
\end{theorem}
\begin{proof}
Use Theorem~\ref{StoJ} with $A_{i}=\partial f_{i}$ for $i=1,2$.
\end{proof}

\begin{theorem}The mapping 
$\prox_{f_{2}/\lambda_{1}}\negthinspace|_E:E\to F$ 
is a homeomorphism with inverse
$\prox_{f_{1}/\lambda_{2}}\negthinspace|_F$. 
\end{theorem}
\begin{proof}
The results follow from Theorem~\ref{evening}\ref{evening6}.
\end{proof}

\begin{corollary}\label{allfunctions}
\begin{enumerate}
\item $E\neq\varnothing$ if and only if $F\neq\varnothing$ if and only if $S\neq\varnothing$ if and only if $\argmin\pfl=\Fix(\lambda_{1}\prox_{f_{1}}+\lambda_{2}\prox_{f_{2}})\neq\varnothing$.
\item
$E \text{ is a singleton}$  if and only if $F \text{ is a singleton}$ if and only if $ S \text{ is a singleton}$ if and only if
$ \argmin\pfl=\Fix(\lambda_{1}\prox_{f_{1}}+\lambda_{2}\prox_{f_{2}}) \text{ is a singleton}.$
\item $\Fix(\lambda_{1}\prox_{f_{1}}+\lambda_{2}\prox_{f_{2}})= \lambda_{1}E+\lambda_{2}F.$
\end{enumerate}
\end{corollary}
\begin{proof}
(i) and (ii): Combine Theorems~\ref{amy1}, \ref{amy2}, \ref{amy3}. (iii):
Apply Corollary~\ref{allsets}(iii) with $A_{i}=\partial f_{i}$ for $i=1, 2$.
\end{proof}

Applying Theorem~\ref{amy3} to $\lambda_{2}f_{1}, \lambda_{1}f_{2}$ gives
\begin{corollary} \label{alternativef1f2}
$$\Fix(\lambda_{1}\prox_{\lambda_{2}f_{1}}+\lambda_{2}\prox_{\lambda_{1}f_{2}})
=\menge{\lambda_{1}x+\lambda_{2}y}{x=\prox_{f_{1}}y, y=\prox_{f_{2}}(x)}.$$
\end{corollary}

\subsection{The case when $\argmin f_1 \cap \argmin f_2\neq\varnothing$}

Note that
\begin{equation}\label{intersection1}
\Fix(\prox_{f_{1}/\lambda_{2}}) =\Fix(\prox_{f_{1}})=\argmin f_{1},\quad
\Fix(\prox_{f_{2}/\lambda_{1}})  =\Fix(\prox_{f_{2}})=\argmin f_{2}.
\end{equation}
\begin{theorem}\label{shrinkfinal}
Assume that $\argmin f_1 \cap \argmin f_2 \neq\varnothing$.
Then
\begin{equation}\label{intersection2}
\Fix(\lambda_1\prox_{f_{1}}+\lambda_2\prox_{f_{2}})=\Fix(\prox_{f_{1}/\lambda_{2}})\cap
\Fix(\prox_{f_{2}/\lambda_{1}})=\Fix(\prox_{f_{1}})\cap \Fix(\prox_{f_{2}}).
\end{equation}
Moreover,
\begin{equation}\label{intersection3}
\Fix(\lambda_1\prox_{f_{1}}+\lambda_2\prox_{f_{2}})=\Fix(\prox_{f_{1}/\lambda_{2}} \prox_{f_{2}/\lambda_{1}})=\Fix(\prox_{f_{2}/\lambda_{1}}\prox_{f_{1}/\lambda_{2}}).
\end{equation}
\end{theorem}
\begin{proof}
Apply Theorem~\ref{t:thesameset}.
\end{proof}

\subsection{Examples on projections}
Projection algorithms, which are instances of the proximal point
algorithm, are important in applications.
Let $C_{1},C_{2}\subseteq \HH$ be nonempty closed convex sets.
With $f_{i}=\iota_{C_{i}}$, \eqref{mixed},
\eqref{meeting1}, \eqref{meeting2}, \eqref{yong} transpire to
\begin{equation}\label{moon0}
\min_{(x,y)}\  g(x,y)
=\bigg(\iota_{C_{1}}(x)+\iota_{C_{2}}(y)+\frac{\|x-y\|^2}{2}\bigg),
\end{equation}
\begin{equation}\label{moon1}
\min_{x}\ g_{1}(x)=\iota_{C_{1}}(x)+\tfrac{1}{2}d_{C_2}^2(x),
\end{equation}
\begin{equation}\label{moon2}
\min_{y}\ g_{2}(y)=\tfrac{1}{2}d_{C_{1}}^2(y)+\iota_{C_{2}}(y),
\end{equation}
\begin{equation}\label{moon3}
\min_{z}\ \frac{g_{3}(z)}{\lambda_{1}\lambda_{2}}=\lambda_{2}^{-1}\tfrac{d_{C_{1}}^2(z)}{2}
+\lambda_{1}^{-1}\tfrac{d_{C_{2}}^2(z)}{2}.
\end{equation}
The Fenchel-Rockafellar dual of \eqref{moon0} given by
\eqref{dualprob} transpires to
$$-\inf_{\phi}\bigg(\sigma_{C_2-C_{1}}(\phi)+\frac{\|\phi\|^2}{2}\bigg),$$
with the unique solution $\bar{\phi}$,
where $\sigma_{C_{2}-C_{1}}(\phi)=
\sup\menge{\langle \phi, y-x\rangle}{ x\in C_{1}, y\in C_{2}}.$
In fact, convex calculus (see, e.g., \cite[Theorem~2.1]{BauBorSVA}) 
or Remark~\ref{r:blablabla} yields
\begin{equation}
\label{e:dean1}
\bar{\phi} = -P_{\overline{C_2-C_1}}(0).
\end{equation}

\begin{theorem}\label{twosets}
Let $C_{1}, C_{2} \subseteq \HH$ be nonempty closed convex sets. Then
\begin{enumerate}
\item
\emph{({\bf Fixed points of alternating projections})}
$S=\argmin g =\menge{(x,y)}{x=P_{C_{1}}(y), y=P_{C_{2}}x}$.
\item \emph{({\bf Fixed points of projection composition})}
$E=\argmin g_{1} =\menge{x}{x=P_{C_{1}}P_{C_{2}}x}$.
\item \emph{({\bf Fixed points of projection composition})}
$F=\argmin g_{2} =\menge{y}{y=P_{C_{2}}P_{C_{1}}y}$.
\item \emph{({\bf Fixed points of the average of projections})}
$\argmin g_{3} =\menge{z}{z=\lambda_{1}P_{C_{1}}(z)+
\lambda_{2}P_{C_{2}}(z)}$.
\end{enumerate}
Moreover,
\begin{enumerate}
\item The mapping $T: S\rightarrow \Fix(\lambda_{1}P_{C_{1}}+
\lambda_{2}P_{C_{2}})$ given by
$$T(x,y)=\lambda_{1}x+\lambda_{2}y,$$
is a homeomorphism with inverse $T^{-1}(z)=(P_{C_{1}}(z),P_{C_{2}}(z))$ for every $z\in \Fix(\lambda_{1}P_{C_{1}}+
\lambda_{2}P_{C_{2}})$.
Hence
$ \Fix(\lambda_{1}P_{C_{1}}+
\lambda_{2}P_{C_{2}})=L(S)$ where
$L:\HH\times \HH\rightarrow\HH: (x,y)\mapsto \lambda_{1}x+\lambda_{2}y$.
\item The mapping $H_{1}:E\rightarrow
\Fix(\lambda_{1}P_{C_{1}}+
\lambda_{2}P_{C_{2}})$ given by
$$H_{1}(x)=\lambda_{1}x+\lambda_{2}P_{C_{2}}x=x-\lambda_{2}\bar{\phi},$$
is a homeomorphism with inverse $H_{1}^{-1}(z)=P_{C_{1}}(z)$ for every $z\in \Fix(\lambda_{1}P_{C_{1}}+
\lambda_{2}P_{C_{2}})$.
Hence $\Fix(\lambda_{1}P_{C_{1}}+
\lambda_{2}P_{C_{2}})=E-\lambda_{2}\bar{\phi}.$
\item The mapping $H_{2}:F\rightarrow
\Fix(\lambda_{1}P_{C_{1}}+
\lambda_{2}P_{C_{2}})$ given by
$$H_{2}(y)=\lambda_{1}P_{C_{1}}(y)+\lambda_{2}y=\lambda_{1}\bar{\phi}+y,$$
is a homeomorphism with inverse $H_{2}^{-1}(z)=P_{C_{2}}(z)$ for every $z\in \Fix(\lambda_{1}P_{C_{1}}+
\lambda_{2}P_{C_{2}})$. Hence
$\Fix(\lambda_{1}P_{C_{1}}+
\lambda_{2}P_{C_{2}})=F+\lambda_{1}\bar{\phi}$.
\item $\Fix(\lambda_{1}P_{C_{1}}+
\lambda_{2}P_{C_{2}})=\lambda_{1}E+\lambda_{2}F$.
\end{enumerate}
\end{theorem}

\begin{theorem} Assume that $C_{1},C_{2}\subseteq\HH$ are two closed convex sets such that
$C_{1}\cap C_{2}\neq\emp$. Then
$$\Fix(\lambda_{1}P_{C_{1}}+\lambda_{2}P_{C_{2}})=\Fix P_{C_{1}}P_{C_{2}}=
\Fix P_{C_{2}}P_{C_{1}}=C_{1}\cap C_{2}.$$
\end{theorem}
\begin{proof}
As $\min g_{3}=\min g_{2}=\min g_{1}=\min g=0$ when $C_{1}\cap C_{2}\neq
\emp$, we have
$$\argmin g_{1}=\argmin g_{2}=\argmin g_{3}=C_{1}\cap C_{2},$$
and $\argmin g=\menge{(x,x)}{x\in C_{1}\cap C_{2}}$. Alternatively, use Theorem~\ref{shrinkfinal}
or Theorem~\ref{t:thesameset}.
\end{proof}

As $\partial \iota_{C}=N_{C}$, $\prox_{\iota_{C}}=P_{C}$,
Theorems~\ref{characterization} and \ref{penticton} give
characterizations of $S, E, F,
\Fix(\lambda_{1}P_{C_{1}}+\lambda_{2}P_{C_{2}})$
in terms of the dual solution $\bar{\phi}$:
\begin{theorem}
\begin{enumerate}
\item $(x,y)\in S$ if and only if
$$-\bar{\phi}\in N_{C_{1}}(x),\quad \bar{\phi}\in N_{C_{2}}(y),\quad \bar{\phi}=x-y.$$
\item $x\in E$ if and only if
$$-\bar{\phi}\in N_{C_{1}}(x),\quad \bar{\phi}=x-P_{C_{2}}(x).$$
\item $y\in F$ if and only if
$$-\bar{\phi}=y-P_{C_{1}}(y),\quad \bar{\phi}\in N_{C_{2}}(y).$$
\item $z\in \Fix(\lambda_{1}P_{C_{1}}+\lambda_{2}P_{C_{2}})$ if and only if
$$-\bar{\phi}= \lambda_{2}^{-1}(z-P_{C_{1}}(z)),\quad
\bar{\phi}= \lambda_{1}^{-1}(z-P_{C_{2}}(z)).$$
\end{enumerate}
\end{theorem}

\section{Algorithms and examples}\label{sanitycheck}
In this section, 
notation is as in section~\ref{keyresultnow} and
we also assume that $\Fix J_{A}\neq \varnothing$. 
By Corollary~\ref{allsets}, $E,F,S$ all are nonempty.
The following results give different algorithms to find
a point in $\Fix J_{A}$.

\begin{theorem}[Fixed point of resolvent average by alternating resolvent
method]~\\ \label{vocation1}
Fix $x_{0}\in\HH$ and for every $n\in\NN$, set
$$y_{n}=J_{A_{2}/\lambda_{1}}x_{n},\qquad x_{n+1}
=J_{A_{1}/\lambda_{2}}y_{n}.$$
Then $\lambda_{1}x_{n}+\lambda_{2}y_{n}\rightharpoonup
\lambda_{1}\overline{x}+\lambda_{2}\overline{y}\in \Fix J_{A}.$
\end{theorem}

\begin{proof} By \cite[Theorem 3.3]{Baus05},
$(x_{n},y_{n})\rightharpoonup (\overline{x},\overline{y})\in S$.
By Theorem~\ref{StoJ},
$\lambda_{1}\bar{x}+\lambda_{2}\bar{y}\in \Fix J_{A}$.
Therefore $\lambda_{1}x_{n}+\lambda_{2}y_{n}\rightharpoonup
\lambda_{1}\overline{x}+\lambda_{2}\overline{y}\in \Fix J_{A}.$
\end{proof}

\begin{theorem}[Fixed point of resolvent average by proximal point
method]~\\ \label{vocation2}
Fix $x_{0}\in\HH$ and for every $n\in\NN$, set
\begin{equation}\label{pp}
x_{n+1}=(\lambda_{1}J_{A_{1}}+\lambda_{2}J_{A_{2}})(x_{n}).
\end{equation}
Then $x_{n}\rightharpoonup \bar{x}\in \Fix J_{A}$.
\end{theorem}
\begin{proof}
As $\lambda_{1}J_{A_{1}}+\lambda_{2}J_{A_{2}}=J_{A}$,
the iteration \eqref{pp}
is the proximal point algorithm.
By Fact~\ref{terry}, $x_{n}\rightharpoonup \bar{x}\in \Fix J_{A}$.
\end{proof}

\begin{theorem}[Fixed point of resolvent average by resolvent
compositions]~\
\label{vocation3}
\begin{enumerate}
\item 
Fix $x_{0}\in\HH$ and for every $n\in\NN$, set
$$x_{n+1}=J_{A_{1}/\lambda_{2}}J_{A_{2}/\lambda_{1}}x_{n}.$$
Then $x_{n}\rightharpoonup x\in 
\Fix J_{A_{1}/\lambda_{2}}J_{A_{2}/\lambda_{1}}$ and 
$\lambda_{1}x_n+ \lambda_{2}J_{A_{2}/\lambda_{1}}x_n
\rightharpoonup 
\lambda_{1}x+ \lambda_{2}J_{A_{2}/\lambda_{1}}x\in \Fix J_{A}$.
\item 
Fix $y_{0}\in\HH$ and for every $n\in\NN$, set
$$y_{n+1}=J_{A_{2}/\lambda_{1}}J_{A_{1}/\lambda_{2}}y_{n}.$$
Then $y_{n}\rightharpoonup y\in
\Fix J_{A_{2}/\lambda_{1}}J_{A_{1}/\lambda_{2}}$
and 
$\lambda_{1}J_{A_{1}/\lambda_{2}}y_n+ \lambda_{2}y_n
\rightharpoonup 
\lambda_{1}J_{A_{1}/\lambda_{2}}y+ \lambda_{2}y\in \Fix J_{A}$.
\end{enumerate}
\end{theorem}
\begin{proof}
(i). 
Since $J_{A_{1}/\lambda_{2}}, J_{A_{2}/\lambda_{1}}$
are firmly nonexpansive, Fact~\ref{bruckreich} shows that
$J_{A_{1}/\lambda_{2}}J_{A_{2}/\lambda_{1}}$ is
strongly nonexpansive and that $x_{n}\rightharpoonup x\in
\Fix J_{A_{1}/\lambda_{2}}J_{A_{2}/\lambda_{1}}$. 
By \cite[Theorem~3.3(iii)]{Baus05}, 
$J_{A_{2}/\lambda_{1}}x_n-x_n\to u^*$,
which implies that
$J_{A_{2}/\lambda_{1}}x_n \rightharpoonup u^*+x$.
Hence 
$\lambda_{1}x_n+ \lambda_{2}J_{A_{2}/\lambda_{1}}x_n
\rightharpoonup
\lambda_1 x + \lambda_2(u^*+x)
= x+ \lambda_2 u^*
= \lambda_1 x + \lambda_2 J_{A_{2}/\lambda_{1}}x
\in \Fix J_A$ by
Theorem~\ref{plumber}\ref{monday1}. 
(ii). The proof is similar to the proof of (i). 
\end{proof}

\begin{remark}(i). Note that when $\HH=\RR^N$, the
weak convergence and norm convergence coincide.
Hence in $\RR^N$, the convergence in
Theorems~\ref{vocation1}, \ref{vocation2}, \ref{vocation3}
is norm convergence.

(ii). As $J_{A}$ (even a projection mapping) need not be weakly
sequentially continuous, one cannot conclude directly that
$\lambda_{1}x_{n}+\lambda_{2}J_{A_{2}/\lambda_{1}}x_{n}\rightharpoonup
\lambda_{1}x+\lambda_{2}J_{A_{2}/\lambda_{1}}x$
in Theorem~\ref{vocation3}(i) or that
$\lambda_{1}J_{A_{1}/\lambda_{2}}y_{n}+
\lambda_{2}y_{n}\rightharpoonup \lambda_{1}J_{A_{1}/\lambda_{2}}y+
\lambda_{2}y$ in Theorem~\ref{vocation3}(ii).
Indeed, following Zarantonello \cite[page~245]{Zara71},
consider the Hilbert sequence space $\ell^{2}$.
Let $\ball:=\menge{x\in {\ell}^{2}}{\|x\|\leq 1}$ and
$(e_{n})_{n\in\NN}$ be the basis vectors, i.e.,
$e_{n}=(\underbrace{0,\cdots, 0, 1}_{\text{$n$ terms}}, 0,\cdots)$. We have
$$e_{1}+e_{n}\rightharpoonup e_{1},\quad P_{\ball}(e_{1})=e_{1},$$
$$(\forall n \geq 2)\ P_{\ball}(e_{1}+e_{n})=
\frac{e_{1}+e_{n}}{\sqrt{2}}\rightharpoonup\frac{e_{1}}{\sqrt{2}}
\neq P_{\ball}(e_{1}).$$
Hence $P_{\ball}$ is not weakly sequentially continuous.
However, in the proof of Theorem~\ref{vocation3}, we invoked
the analysis in \cite[Section~3.2]{Baus05} which allowed us to
obtain the weak convergence conclusion. 
\end{remark}

We end with three examples to illustrate our main results. 

\begin{example} Consider
$$C_{1}=\menge{(x,y)}{ x^2+(y-2)^2\leq 1},\quad
C_{2}=\menge{(x,0)}{ x\in\RR}.$$
Then $C_{1}\cap C_{2}=\varnothing$.
\emph{We claim that when $\lambda_{1}+\lambda_{2}=1$ with $\lambda_{i}> 0$,
$$\Fix(\lambda_{1}P_{C_{1}}+\lambda_{2}P_{C_{2}})=\{(0,\lambda_{1})\}.$$}
In this example, $\Fix P_{C_{1}}P_{C_{2}}$ is easier to compute
than $\Fix(\lambda_{1}P_{C_{1}}+\lambda_{2}P_{C_{2}})$.
Indeed, we have
\begin{align*} P_{C_{1}}(x,y) &
=\begin{cases}
\left(\displaystyle \frac{x}{\sqrt{x^2+(y-2)^2}},
\frac{y-2}{\sqrt{x^2+(y-2)^2}}+2\right), & \text{ if $(x,y)\not\in
C_{1}$}\\[+4mm]
(x,y), & \text{ if $(x,y)\in C_{1}$},
\end{cases}
\\[+4mm]
P_{C_{2}}(x,y) & =(x,0).
\end{align*}
Thus,
\begin{align*}
P_{C_{1}}P_{C_{2}}(x,y) &= P_{C_{1}}(x,0)\\
&=\left(\frac{x}{\sqrt{x^2+4}},\frac{-2}{\sqrt{x^2+4}}+2\right),
\end{align*}
since $(x,0)\not\in C_{1}$.
Start with $(x_{0},y_0)$. Consider
the composition algorithm $(x_{n+1},y_{n+1})=P_{C_{1}}P_{C_{2}}(x_{n},y_{n})$. We have
$$(x_{n+1},y_{n+1})=\left(\frac{x_{n}}{\sqrt{x_{n}^2+4}},\frac{-2}{\sqrt{x_{n}^2+4}}+2\right).$$
It follows that
$$|x_{n+1}|=\frac{|x_{n}|}{\sqrt{x_{n}^2+4}}\leq \frac{|x_{n}|}{2}\leq\cdots\leq \frac{|x_{0}|}{2^{n+1}},$$
and this gives $x_{n+1}\rightarrow 0$, consequently $y_{n+1}\rightarrow 1$. Therefore,
$$(0,1)\in \Fix P_{C_{1}}P_{C_{2}}.$$
In fact, by using
$$(x,y)=\left(\frac{x}{\sqrt{x^2+4}},\frac{-2}{\sqrt{x^2+4}}+2\right),$$
we see that $(x,y)=(0,1)$. Hence $\Fix P_{C_{1}}P_{C_{2}}=\{(0,1)\}$. Therefore, by Theorem~\ref{twosets}(ii)
\begin{equation}\label{friday1}
\Fix(\lambda_{1}P_{C_{1}}+\lambda_{2}P_{C_{2}})=\lambda_{1}(0,1)
+\lambda_{2}P_{C_{2}}(0,1)=(0,\lambda_{1}),
\end{equation}
since $P_{C_{2}}(0,1)=(0,0)$.

For $P_{C_{2}}P_{C_{1}}$, since $P_{C_{2}}P_{C_{1}}(x_{1},y_{1})\in C_{2}$ (not in $C_1$) we have
$$(x_{n+1},y_{n+1})=P_{C_{2}}P_{C_{1}}(x_{n},y_{n})=\bigg(\frac{x_{n}}{\sqrt{x_{n}^2+4}},0\bigg)\quad \forall \ n\geq 2,$$
and $\Fix(P_{C_{2}}P_{C_{1}})=\{(0,0)\}.$ Then for $(0,0)\in \Fix(P_{C_{2}}P_{C_{1}})$,
\begin{equation}\label{friday2}
\lambda_{1}P_{C_{1}}(0,0)+\lambda_{2}(0,0)=\lambda_{1}(0,1)=(0,\lambda_{1}),
\end{equation}
which shows also that
$\Fix(\lambda_{1}P_{C_{1}}+\lambda_{2}P_{C_{2}})=\{(0,\lambda_{1})\}.$

On the other hand, the averaged projection method proceeds as follows.
\begin{align}
(\lambda_{1}P_{C_{1}}+\lambda_{2}P_{C_{2}})(x,y) & =
 \begin{cases}
\lambda_{1}\left(\displaystyle
\frac{x}{\sqrt{x^2+(y-2)^2}},\frac{y-2}{\sqrt{x^2+(y-2)^2}}+2\right)+\lambda_{2}(x,0),
& \text{ if $(x,y)\not\in C_{1}$;}\\[+4mm]
\lambda_{1}(x,y)+\lambda_{2}(x,0), & \text{ if $(x,y)\in C_{1}$},
\end{cases}\nonumber\\
(x_{n+1},y_{n+1}) & =(\lambda_{1}P_{C_{1}}+\lambda_{2}P_{C_{2}})
(x_{n},y_{n}). \label{trueform}
\end{align}
Start with any $(x_{0},y_{0})\in\RR^2$.

\noindent {\sl Claim: If $n$ is sufficiently large, then $y_{n}<2$ and
$(x_{n},y_{n})\not\in C_{1}$.}

\noindent  To see that, for $y_{n}\geq 2$ consider two cases:
if $(x_{n},y_{n})\in C_{1}$, then $y_{n}\geq 1$,
and $$(x_{n+1},y_{n+1})=
\lambda_{1}(x_{n},y_{n})+\lambda_{2}(x_{n},0)=(x_{n},\lambda_{1}y_{n}),$$
which gives $y_{n+1}=\lambda_{1}y_{n}$; if $(x_{n},y_{n})\not\in C_{1}$, then
$\sqrt{x_{n}^2+(y_{n}-2)^2}\geq 1$ and
$$(x_{n+1},y_{n+1})=\lambda_{1}\left(\frac{x_{n}}{\sqrt{x_{n}^2+(y_{n}-2)^2}},
\frac{y_{n}-2}{\sqrt{x_{n}^2+
(y_{n}-2)^2}}+2\right)+\lambda_{2}(x_{n},0),$$
so that
$$y_{n+1}=\lambda_{1}\left( \frac{y_{n}-2}{\sqrt{x_{n}^2+
(y_{n}-2)^2}}+2\right)\leq \lambda_{1}(y_{n}-2+2)=\lambda_{1}y_{n}.$$
Furthermore, whenever $(x_{n},y_{n})\in C_{1}$,
we have $y_{n+1}=\lambda_{1}y_{n}$ and $y_{n}\geq 1$.
Altogether, $y_{n+1}\leq \lambda_{1}y_{n}$.
This implies that the averaged projection iterations can only
stay in $C_{1}$ for only a finite number of times and that for $n$ sufficiently large
$y_{n}<2$.
Hence for all $n$ sufficiently large,
the average projection algorithm \eqref{trueform}
gives $y_{n}<2$ and
$$(x_{n+1},y_{n+1})=\left(\lambda_{1}\frac{x_{n}}{\sqrt{x_{n}^2+(y_{n}-2)^2}}+\lambda_{2}x_{n},
\lambda_{1}\bigg(\frac{y_{n}-2}{\sqrt{x_{n}^2+(y_{n}-2)^2}}+2\bigg)\right).
$$
Moreover, as for $n$ sufficiently large
$$1\leq \frac{y_{n}-2}{\sqrt{x_{n}^2+(y_{n}-2)^2}}+2\leq 2,$$
we have $\lambda_{1}\leq y_{n+1}\leq \lambda_{1}2<2$.
Then $(x,y)\in\Fix(\lambda_{1}P_{C_{1}}+\lambda_{2}
P_{C_{2}})$ means
$$(x,y)=\left(\lambda_{1}\frac{x}{\sqrt{x^2+(y-2)^2}}+\lambda_{2}x,
\lambda_{1}\bigg(\frac{y-2}{\sqrt{x^2+(y-2)^2}}+2\bigg)\right),
$$
which gives only one solution
$(x,y)=(0,\lambda_{1})$ in view of $\lambda_{1}>0, \lambda_{1}\leq y<2$.
Again this shows that
$$\Fix(\lambda_{1}P_{C_{1}}+\lambda_{2}
P_{C_{2}})=\{(0,\lambda_{1})\},$$
which is consistent with the results given
by \eqref{friday1} and \eqref{friday2}.
(See also \cite[Example 5.3]{baus1994} for more
on the rate of convergence of alternating projections.)
\end{example}

Now what can one say about the relationships among
$\Fix(\lambda_{1}\prox_{f_{1}}+\lambda_{2}\prox_{f_{2}})$,
$ \Fix(\prox_{f_{1}}\prox_{f_{2}}),$ $\Fix(\prox_{f_{2}}\prox_{f_{1}})$ ?
It is tempting to conjecture that for fixed points of
alternating iterations:
$$x=\prox_{f_{1}}y, \quad y=\prox_{f_{2}}x, $$
one has $\lambda_{1}x+\lambda_{2}y\in
\Fix(\lambda_{1}\prox_{f_{1}}+\lambda_{2}\prox_{f_{2}})$
--- and this is true for projections --- but
this is not right in general, as the following examples show.

\begin{example}
Consider $f_{1}(x)=x^2, f_{2}(x)=(x-1)^2$ for $x\in \RR$. Let $\lambda_{1}+\lambda_{2}=1$ with
$\lambda_{i}>0$. Then $\argmin f_{1}=\{0\}, \argmin f_{2}=\{1\},$ so $\argmin f_{1}\cap\argmin f_{2}=\varnothing$. As
$\nabla f_{1}(x)=2x, \nabla f_{1}(x)/\lambda_{2}=2x/\lambda_{2}, \nabla f_{2}(x)=2(x-1), \nabla f_{2}(x)/\lambda_{1}=
2(x-1)/\lambda_{1}$, for every $z\in \RR$,
$$\prox_{f_{1}}(z)=\frac{z}{3},\quad  \prox_{f_{2}}(z)=\frac{z+2}{3},$$
$$\prox_{f_{1}/\lambda_{2}}(z)=\frac{\lambda_{2}z}{2+\lambda_{2}},\quad \prox_{f_{2}/\lambda_{1}}(z)=
\frac{\lambda_{1}z+2}{2+\lambda_{1}}.$$
Moreover, $$x\mapsto \prox_{f_{1}/\lambda_{2}}\prox_{f_{2}/\lambda_{1}}(x)=\frac{\lambda_{2}}{2+\lambda_{2}}
\frac{\lambda_{1}x+2}{2+\lambda_{1}},$$
$$y\mapsto \prox_{f_{2}/\lambda_{1}}\prox_{f_{1}/\lambda_{2}}(y)=\frac{1}{2+\lambda_{1}}\bigg(
\frac{\lambda_{1}\lambda_{2}y}{2+\lambda_{2}}+2\bigg),$$
$$z\mapsto (\lambda_{1}\prox_{f_{1}}+\lambda_{2}\prox_{f_{2}})(z)=\frac{z}{3}+\frac{2\lambda_{2}}{3}.$$
We have
\begin{align*} \Fix (\lambda_{1}\prox_{f_{1}}+\lambda_{2}\prox_{f_{2}})& =\{\lambda_{2}\},\\
\Fix(\prox_{f_{1}/\lambda_{2}}\prox_{f_{2}/\lambda_{1}}) &=\{\lambda_{2}/3\},\\
\Fix(\prox_{f_{2}/\lambda_{1}}\prox_{f_{1}/\lambda_{2}}) &=\{(2+\lambda_{2})/3\},\\
S=\{(x,y)|\ x=\prox_{f_{1}/\lambda_{2}}(y), y=\prox_{f_{2}/\lambda_{1}}(x)\}& =\{(\lambda_{2}/3, (\lambda_{2}
+2)/3)\}.
\end{align*}
As in Theorem~\ref{amy1}, for $x\in E=\{\lambda_{2}/3\}$,
$$
\lambda_{1}x+\lambda_{2}\prox_{f_{2}/\lambda_{1}}(x) = \lambda_{1}\lambda_{2}/3+\lambda_{2}
\frac{\lambda_{1}\lambda_{2}/3+2}{2+\lambda_{1}}=\lambda_{2}\in \Fix (\lambda_{1}\prox_{f_{1}}+\lambda_{2}\prox_{f_{2}}).$$
As in Theorem~\ref{amy2}, for $y\in F=\{(\lambda_{2}+2)/3\},$
$$\lambda_{1}\prox_{f_{1}/\lambda_{2}}(y)+\lambda_{2}y =\lambda_{1}\frac{\lambda_{2}}{2+\lambda_{2}}
\frac{2+\lambda_{2}}{3}+\lambda_{2}\frac{2+\lambda_{2}}{3}=\lambda_{2}\in \Fix (\lambda_{1}\prox_{f_{1}}+\lambda_{2}\prox_{f_{2}}).$$
As in Theorem~\ref{amy3}, for $z\in \Fix (\lambda_{1}\prox_{f_{1}}+\lambda_{2}\prox_{f_{2}})=\{\lambda_{2}\}$,
$$T^{-1}(\lambda_{2})=(\prox_{f_{1}}(\lambda_{2}),
\prox_{f_{2}}(\lambda_{2}))=(\lambda_{2}/3, (\lambda_{2}+2)/3)\in S.$$
As in Theorem~\ref{characterization}, the dual solution satisfies
$$-\bar{\phi}=y-x=2/3=\prox_{f_{2}/\lambda_{1}}(x)-x
=y-\prox_{f_{1}/\lambda_{2}}(y),$$ for $(x,y)\in S$.

We now show that for fixed points of alternating iterations
$$x=\prox_{f_{1}}y, \quad y=\prox_{f_{2}}x,
$$
one has $\lambda_{1}x+\lambda_{2}y \notin
\Fix(\lambda_{1}\prox_{f_{1}}+\lambda_{2}\prox_{f_{2}})$.
Indeed, as
$$x\mapsto \prox_{f_{1}}\prox_{f_{2}}(x)=\frac{x+2}{9},$$
$\Fix(\prox_{f_{1}}\prox_{f_{2}})=\{1/4\}$. With $x=1/4$, we have
$$\lambda_{1}x+\lambda_{2}\prox_{f_{2}}(x)=\lambda_{1}1/4+\lambda_{2}\frac{1/4+2}{3}=
1/4+\lambda_{2}/2\not=\lambda_{2}\in \Fix (\lambda_{1}\prox_{f_{1}}+\lambda_{2}\prox_{f_{2}}),$$
unless $\lambda_{2}=1/2$.
Similarly, one can also show that
$$y\mapsto \prox_{f_{2}}\prox_{f_{1}}(y)=\frac{y+6}{9},$$
has $\Fix (\prox_{f_{2}}\prox_{f_{1}})=\{3/4\}$. With $y=3/4$,
$$\lambda_{1}\prox_{f_{1}}(y)+\lambda_{2}y=1/4+\lambda_{2}/2\neq \lambda_{2}\in \Fix (\lambda_{1}\prox_{f_{1}}+\lambda_{2}\prox_{f_{2}}),$$
unless $\lambda_{2}=1/2$.

However, for $\lambda_{2}f_{1}(x)=\lambda_{2}x^2$, $\lambda_{1}f_{2}(x)=\lambda_{1}(x-1)^2$ with
$$\prox_{\lambda_{2}f_{1}}(x)=\frac{x}{2\lambda_{2}+1}, \quad \prox_{\lambda_{1}f_{2}}(x)=\frac{x+2\lambda_{1}}{2\lambda_{1}+1},$$
for every $x\in \RR$, we have
$$(\lambda_{1}\prox_{\lambda_{2}f_{1}}+\lambda_{2}\prox_{\lambda_{1}f_{2}})(x)=
\lambda_{1}\frac{x}{2\lambda_{2}+1}+\lambda_{2}\frac{x+2\lambda_1}{2\lambda_{1}+1}.$$
By solving
$$x=\lambda_{1}\frac{x}{2\lambda_{2}+1}+\lambda_{2}\frac{x+2\lambda_1}{2\lambda_{1}+1},$$ one indeed has
$$\Fix(\lambda_{1}\prox_{\lambda_{2}f_{1}}+\lambda_{2}\prox_{\lambda_{1}f_{2}})=\{1/4+\lambda_{2}/2\}
=\menge{\lambda_{1}x+\lambda_{2}y}{x=\prox_{f_{1}}y, y=\prox_{f_{2}}(x)},$$
as Theorem~\ref{amy3} or Corollary~\ref{alternativef1f2} shows.
\end{example}

\begin{example} Consider
$f_{1}(x)=|x|, f_{2}(x)=(x-1)^2$ for every $x\in \RR$ and $\lambda_{1}+\lambda_{2}=1$
with $\lambda_{i}>0$. Then $\argmin f_{1}=\{0\}, \argmin f_{2}=\{1\},$ so $\argmin f_{1}\cap\argmin f_{2}=\varnothing$.
As
$$\partial f_1(x)=\begin{cases}
1 & \text{ if $x>0$},\\
[-1,1] & \text{ if $x=0$},\\
-1 & \text{ if $x<0$},
\end{cases}
$$
we have
$$\prox_{f_{1}}(x)=\begin{cases}
x-1 &\text{ if $x>1$},\\
0 & \text{ if $-1\leq x\leq 1$},\\
x+1 &\text{ if $x<-1$},
\end{cases}
$$
$$\prox_{f_{2}}(x)=\frac{x+2}{3}.$$
Then
$$(\lambda_{1}\prox_{f_{1}}+\lambda_{2}\prox_{f_{2}})(x)
=\begin{cases}
\lambda_{1}(x-1)+\frac{\lambda_{2}(x+2)}{3} & \text{ if $x>1$},\\
\frac{\lambda_{2}(x+2)}{3} &\text{ if $-1\leq x\leq 1$},\\
\lambda_{1}(x+1)+\frac{\lambda_{2}(x+2)}{3} &\text{ if $x<-1$},
\end{cases}
$$
and
$$\Fix(\lambda_{1}\prox_{f_{1}}+\lambda_{2}\prox_{f_{2}})=\left\{\frac{2\lambda_{2}}{2+\lambda_{1}}\right\}.
$$
Moreover,
\begin{align*}
\prox_{f_{1}}\prox_{f_{2}}(x) &= \prox_{f_{1}}((x+2)/3)\\
&=\begin{cases}
\frac{x-1}{3} & \text{ if $x>1$},\\
0  & \text{ if $-5\leq x\leq 1$},\\
\frac{x+5}{3} &\text{ if $x<-5$},
\end{cases}
\end{align*}
and $$\Fix(\prox_{f_{1}}\prox_{f_{2}})=\{0\}.$$
For $x\in E=\Fix(\prox_{f_{1}}\prox_{f_{2}}),$ i.e., $x=0$,
\begin{align*}
\lambda_{1}x+\lambda_{2}\prox_{f_{2}}(x) & =\lambda_{1}0+\lambda_{2}\prox_{f_{2}}(0)\\
&= \frac{2\lambda_{2}}{3}\not\in \Fix(\lambda_{1}\prox_{f_{1}}+\lambda_{2}\prox_{f_{2}}).
\end{align*}
Therefore, one cannot use the fixed point set of
$$x=\prox_{f_{1}}y, \quad y=\prox_{f_{2}}x,$$
to recover $\Fix(\lambda_{1}\prox_{f_{1}}+\lambda_{2}\prox_{f_{2}})$.

Now let us consider $\prox_{f_{1}/\lambda_{2}}, \prox_{f_{2}/\lambda_{1}}$. As
$$\partial f_1(x)/\lambda_{2}=\begin{cases}
1/\lambda_{2} & \text{ if $x>0$,}\\
[-1/\lambda_{2},1/\lambda_{2}] & \text{ if $x=0$,}\\
-1/\lambda_{2} & \text{ if $x<0$},
\end{cases}
$$
$$\partial f_{2}(x)/\lambda_{1}=\frac{2(x-1)}{\lambda_{1}},
$$
we have
$$\prox_{f_{1}/\lambda_{2}}(x)=\begin{cases}
x-1/\lambda_{2} &\text{ if $x>1/\lambda_{2}$},\\
0 & \text{ if $-1/\lambda_{2}\leq x\leq 1/\lambda_{2}$},\\
x+1/\lambda_{2} &\text{ if $x<-1/\lambda_{2}$},
\end{cases}
$$
$$ \prox_{f_{2}/\lambda_{1}}(x)=\frac{\lambda_{1}x+2}{2+\lambda_{1}}.
$$
It follows that
\begin{align*}
\prox_{f_{1}/\lambda_{2}}\prox_{f_{2}/\lambda_{1}}(x) &=\prox_{f_{1}/\lambda_{2}}((\lambda_{1}x+2)
/(2+\lambda_{1}))\\
&=\begin{cases}
\frac{\lambda_{1}x+2}{2+\lambda_{1}}-\frac{1}{\lambda_{2}}
& \text{ if $x>\frac{3}{\lambda_{2}},$}\\
0 & \text{ if $\frac{-(3+\lambda_{2})}{\lambda_{1}\lambda_{2}}\leq x\leq \frac{3}{\lambda_{2}},$}\\
\frac{\lambda_{1}x+2}{2+\lambda_{1}}+\frac{1}{\lambda_{2}}
& \text{ if $x<\frac{-(3+\lambda_{2})}{\lambda_{1}\lambda_{2}}$,}
\end{cases}
\end{align*}
and that
$$\Fix (\prox_{f_{1}/\lambda_{2}}\prox_{f_{2}/\lambda_{1}})=\{0\}.$$
For $x\in E=\Fix (\prox_{f_{1}/\lambda_{2}}\prox_{f_{2}/\lambda_{1}})$, i.e., $x=0$, we have
$$\lambda_{1}0+\lambda_{2}\prox_{f_{2}/\lambda_{1}}(0)=\lambda_{2}\frac{\lambda_{1}0+2}{2+\lambda_{1}}=
\frac{2\lambda_{2}}{2+\lambda_{1}}\in
\Fix(\lambda_{1}\prox_{f_{1}}+\lambda_{2}\prox_{f_{2}}).$$
Again, these testify that
$$\Fix(\lambda_{1}\prox_{f_{1}}+\lambda_{2}\prox_{f_{2}})=\{\lambda_{1}x+\lambda_{2}\prox_{f_{2}/\lambda_{1}}
(x)|\; x\in E\}.$$
Similarly, one can verify that
\begin{align*}
\prox_{f_{2}/\lambda_{1}}\prox_{f_{1}/\lambda_{2}}(x)=
&\begin{cases}
\frac{\lambda_{1}(x-1/\lambda_{2})+2}{2+\lambda_{1}}
& \text{ if $x>1/\lambda_{2}$},\\
 \frac{2}{2+\lambda_{1}} & \text{ if $-1/\lambda_{2}\leq x \leq 1/\lambda_{2}$},\\
\frac{\lambda_{1}(x+1/\lambda_{2})+2}{2+\lambda_{1}}
& \text{ if $x<-1/\lambda_{2}$,}
\end{cases}
\end{align*}
and
$$F=\Fix(\prox_{f_{2}/\lambda_1}\prox_{f_{1}/\lambda_{2}})=\left\{\frac{2}{2+\lambda_{1}}\right\}.$$
For $y\in F$, i.e., $0< y=2/(2+\lambda_{1})<1$,
$$\lambda_{1}\prox_{f_{1}/\lambda_{2}}(y)+\lambda_{2}y=\lambda_{1}0+\lambda_{2}\frac{2}{(2+\lambda_{1})}
=\frac{2\lambda_{2}}{2+\lambda_{1}}\in \Fix (\lambda_{1}\prox_{f_{1}}+\lambda_{2}\prox_{f_{2}}).
$$
Again,
$$\Fix(\lambda_{1}\prox_{f_{1}}+\lambda_{2}\prox_{f_{2}})=\{\lambda_{1}\prox_{f_{1}/\lambda_{2}}(y)
+\lambda_{2}y|\
|\; y\in F\}.$$
Finally, since $$S=\{(x,\prox_{f_{2}/\lambda_{1}}(x))|\ x\in E\}=\left\{\bigg(0,\frac{2}{2+\lambda_{1}}\bigg)\right\},$$
we have $\bar{\phi}=-2/(2+\lambda_{1}).$
\end{example}
More examples can be constructed by using the proximal mapping calculus developed by Combettes and Wajs \cite{patrick2}.

\begin{remark}
We conclude by pointing out that the situation for three or more functions
(or sets if we work with indicator functions) is not clear at the moment.
For instance, as pointed out by De Pierro and attributed to Iusem
\cite{DePierro},
one may have 3 sets such that least squares solutions exist but the
existence of fixed points of compositions depends on the \emph{order}
of the projections.
To fully understand these situations is an interesting topic for further
research.
\end{remark}

\section*{Acknowledgments}
Xianfu Wang was partially
supported by the Natural Sciences and Engineering Research Council
of Canada.
Heinz Bauschke was partially supported by the Natural Sciences and
Engineering Research Council of Canada and by the Canada Research Chair
Program.

\end{document}